\documentclass[11pt,twoside]{amsart}
\usepackage{a4wide}
\usepackage{amssymb}
\usepackage{amsmath}
\usepackage{amsthm}
\usepackage{xypic}

\newcommand{\im}{\operatorname{im}}
\newcommand{\id}{\operatorname{id}}
\newcommand{\pr}{\operatorname{pr}}

\newcommand{\Hom}{\operatorname{Hom}}
\newcommand{\Lie}{\operatorname{Lie}}

\newcommand{\tors}{\operatorname{tors}}

\theoremstyle{plain}

\newtheorem{theorem}{Theorem}[section]
\newtheorem{corollary}[theorem]{Corollary}
\newtheorem{lemma}[theorem]{Lemma}
\newtheorem{remark}[theorem]{Remark}
\newtheorem{proposition}[theorem]{Proposition}

\theoremstyle{remark}

\newtheorem{remar}[theorem]{Remark}

\newtheorem*{rem}{Remark}

\begin{document}

\title[]{$SK_1$ and Lie algebras}%
\author{Peter Schneider and Otmar Venjakob}%
\address{Universit\"{a}t M\"{u}nster,  Mathematisches Institut,  Einsteinstr. 62,
48291 M\"{u}nster,  Germany,
 http://www.uni-muenster.de/math/u/schneider/ }%
\email{pschnei@math.uni-muenster.de }%

\address{Universit\"{a}t Heidelberg,  Mathematisches Institut,  Im Neuenheimer Feld 288,  69120
Heidelberg,  Germany,
 http://www.mathi.uni-heidelberg.de/$\,\tilde{}\,$venjakob/}
\email{venjakob@mathi.uni-heidelberg.de}

\thanks{ }%

\date{July 26, 2011}%
\begin{abstract}
We investigate the vanishing of the group $SK_1(\Lambda(G))$ for the Iwasawa algebra $\Lambda(G)$ of a pro-$p$ $p$-adic Lie group $G$ (with $p \neq 2$). We reduce this vanishing to a linear algebra problem for Lie algebras over arbitrary rings, which we solve for Chevalley orders in split reductive Lie algebras.
\end{abstract}
\maketitle

\section*{Introduction}

Throughout the paper we fix a prime numer $p \neq 2$, and we always let $G$ be a pro-$p$ $p$-adic Lie group with Lie algebra $\Lie_{\mathbb{Q}_p}(G)$. Its Iwasawa algebra is defined to be $\Lambda(G) := \varprojlim_N \mathbb{Z}_p [G/N]$ where $N$ run over all open normal subgroups of $G$. Correspondingly we have the larger algebra $\Lambda^\infty(G) := \varprojlim_N \mathbb{Q}_p [G/N]$. The inclusion $\Lambda(G) \subseteq \Lambda^\infty(G)$ induces a homomorphism of algebraic $K$-groups, and we define
\begin{equation*}
    SK_1(\Lambda(G)) := \ker \big( K_1(\Lambda(G)) \longrightarrow K_1(\Lambda^\infty(G)) \big) \ .
\end{equation*}
This generalizes the well known definition of $SK_1(\mathbb{Z}[G])$ for a finite $p$-group $G$ as, for example, in \cite{Oli}. The problem which we address in this paper is the vanishing of $SK_1(\Lambda(G))$. In the finite group case this does not seem to have a simple answer beyond abelian groups. In contrast our main result Thm.\ \ref{main} says that any principal congruence subgroup $G$ of a split reductive algebraic group over $\mathbb{Q}_p$ (at least if $p \geq 5$) satisfies $SK_1(\Lambda(G)) = 0$. Such congruence subgroups all have the property of being uniform pro-$p$-groups. Although we will heavily make use of this fact, we also will see in section \ref{s:counterex} that there are uniform groups with nonvanishing $SK_1$.

Our interest in this problem comes from noncommutative Iwasawa theory where $G$ is a Galois group of some (infinite) extension of number fields. Cohomological invariants of arithmetic objects often are related to classes in $K_1(\Lambda(G))$. On the other hand the group $K_1(\Lambda^\infty(G))$ can be identified with the group of Galois equivariant $\overline{\mathbb{Z}}_p^\times$-valued functions on the set of isomorphism classes of $\overline{\mathbb{Q}}_p$-irreducible Artin representations of $G$ (cf.\ \cite{SV} \S3). Important such functions are given by $L$-values. One of the basic questions of Iwasawa theory is to which extent $L$-values determine global cohomological invariants. This obviously is related to the injectivity of the above map whose kernel was defined to be $SK_1(\Lambda(G))$. In particular, the vanishing of $SK_1(\Lambda(G))$ is a necessary condition for $p$-adic $L$-functions to be uniquely defined (cf.\ \cite{Kak}).

For a general uniform pro-$p$-group $G$ Lazard (\cite{Laz}) has constructed a corresponding Lie algebra $\mathcal{L}(G)$ over $\mathbb{Z}_p$ which is a lattice in the $\mathbb{Q}_p$-Lie algebra of $G$ in the sense of $p$-adic Lie groups. Generalizing a homological description of $SK_1$ of a finite group by Oliver in \cite{Oli} and using another important result by Lazard, that the cohomology algebra $H^\ast(G, \mathbb{F}_p)$ of any uniform $G$ is the exterior algebra on $H^1(G, \mathbb{F}_p)$, as well as the commutator calculus in $G$ (cf.\ \cite{GS}) we translate our problem in Thm.\ \ref{Lie-criterion1}  to a purely Lie theoretic property of $\mathcal{L}(G)$.

In fact this Lie theoretic property makes sense for any Lie algebra $\mathcal{L}$ over any ring $R$ which is finitely generated free as an $R$-module. It is the question whether the kernel of the Lie bracket viewed as a linear map $\bigwedge^2 \mathcal{L} \longrightarrow \mathcal{L}$ is generated by indecomposable elements in the kernel. To our knowledge the only situation where this has been studied in the literature is the case of semisimple complex Lie algebras. Here Kostant (\cite{Ko}) showed by representation theoretic methods that the question always has a positive answer. We will in fact show that the answer is positive for any Chevalley order of a split reductive Lie algebra base changed to any ring in which 2 and 3 are invertible. In such a general setting representation theoretic tools do not apply anymore. We use instead the combinatorics of the root subspaces. In this sense our approach is much more elementary than the one by Kostant. We also show that the assumption that the ambient Lie algebra is reductive is not necessary. The answer also is positive for the nilpotent radical of any Borel subalgebra in the above Chevalley orders.

By a different line of reasoning which uses the Lazard isomorphism for continuous group cohomology and the classical Whitehead lemmata about the cohomology of semisimple Lie algebras we show in Thm.\ \ref{finiteness} that $SK_1(\Lambda(G))$ is finite for any pro-$p$ $p$-adic Lie group $G$ whose Lie algebra divided the Lie algebra of the center of $G$ is semisimple.

\textit{General notation:} Let $\tors M$ and $M[p]$, for an abelian group $M$, denote the torsion subgroup and the subgroup of elements killed by $p$, respectively. If $M$ is an abelian pro-$p$ group then $M^\vee := \Hom^{cont}(M,\mathbb{Q}_p/\mathbb{Z}_p)$
denotes its Pontrjagin dual consisting of all continuous group homomorphisms from $M$ into $\mathbb{Q}_p/\mathbb{Z}_p$.

\section{Preliminaries}

In this section we give a description of $SK_1(\Lambda(G))$ in terms of group cohomology. This is a straightforward consequence of results in \cite{Oli} in the finite group case.

Let $H$ be any finite $p$-group. We have $SK_1(\mathbb{Z}_p[H]) = \ker \big( K_1(\mathbb{Z}_p[H]) \longrightarrow K_1(\mathbb{Q}_p[H]) \big)$.
According to \cite{Oli} Prop.\ 8.4
and Thm.\ 8.6 there is an exact sequence
\begin{equation*}
    \oplus_{A \subseteq H} H_2(A,\mathbb{Z}) \longrightarrow
    H_2(H,\mathbb{Z}) \longrightarrow SK_1(\mathbb{Z}_p[H])
    \longrightarrow 0
\end{equation*}
where $A$ runs over all abelian subgroups of $H$. All three terms in
this sequence are finite abelian $p$-groups. Since
$\mathbb{Q}_p/\mathbb{Z}_p$ is an injective abelian group we have,
by the universal coefficient theorem,
\begin{equation*}
    H^2(H,\mathbb{Q}_p/\mathbb{Z}_p) =
    \Hom(H_2(H,\mathbb{Z}), \mathbb{Q}_p/\mathbb{Z}_p)
\end{equation*}
and hence the dual exact sequence
\begin{equation}\label{f:finite}
    0 \longrightarrow SK_1(\mathbb{Z}_p[H])^\vee \longrightarrow
    H^2(H,\mathbb{Q}_p/\mathbb{Z}_p) \xrightarrow{\; {\rm res} \;} \prod_{A \subseteq H}
    H^2(A,\mathbb{Q}_p/\mathbb{Z}_p) \ .
\end{equation}

We now let $H$ run over the factor groups $G/N$ of $G$ by open
normal subgroups $N$ and pass to the projective limit in
\eqref{f:finite}. Since
\begin{equation*}
    SK_1(\Lambda(G)) = \varprojlim_N SK_1(\mathbb{Z}_p[G/N])
\end{equation*}
by \cite{SV} Cor.\ 3.2 we obtain an isomorphism
\begin{equation*}
     SK_1(\Lambda(G))^\vee = \ker\big(
     H^2(G,\mathbb{Q}_p/\mathbb{Z}_p)
     \longrightarrow \varinjlim_N \prod_{A \subseteq G/N}
     H^2(A,\mathbb{Q}_p/\mathbb{Z}_p) \big) \ .
\end{equation*}

\begin{lemma}\label{limit-kernel}
We have
\begin{multline*}
   \ker\big(
     H^2(G,\mathbb{Q}_p/\mathbb{Z}_p)
     \longrightarrow \varinjlim_N \prod_{A \subseteq G/N}
     H^2(A,\mathbb{Q}_p/\mathbb{Z}_p) \big) \\
  = \ker\big(  H^2(G,\mathbb{Q}_p/\mathbb{Z}_p)
     \xrightarrow{\; {\rm res} \;} \prod_{A \subseteq G}
     H^2(A,\mathbb{Q}_p/\mathbb{Z}_p) \big)
\end{multline*}
where on the right hand side $A$ runs over all closed abelian subgroups of $G$.
\end{lemma}
\begin{proof}
Let $c \in H^2(G,\mathbb{Q}_p/\mathbb{Z}_p)$ be any fixed cohomology class. First suppose that $c$ does not lie in the right hand kernel. We then find a closed abelian subgroup $A \subseteq G$ such that $c|A \neq 0$. Whenever $N \subseteq G$ is an open normal subgroup such that $c = \mathrm{inf}(c')$ lifts to a class $c' \in H^2(G/N,\mathbb{Q}_p/\mathbb{Z}_p)$ we must have $c'|(AN/N) \neq 0$. Hence $c$ does not lie in the left hand kernel either.

Vice versa, let us now assume that $c$ does not lie in the left hand kernel. We find an open normal subgroup $N_1 \subseteq G$ such that $c = \mathrm{inf}(c_1)$ lifts to some $c_1 \in H^2(G/N_1,\mathbb{Q}_p/\mathbb{Z}_p)$. We now choose a decreasing sequence $N_1 \supseteq N_2 \supseteq \ldots$ of open normal subgroups $N_j \subseteq G$ such that $\bigcap_j N_j = 1$, and we put $c_j := \mathrm{inf}(c_1) \in H^2(G/N_j,\mathbb{Q}_p/\mathbb{Z}_p)$. Our assumption on $c$ guarantees that the set $\mathfrak{A}_j$ of all abelian subgroups $A \subseteq G/N_j$  such that $c_j|A \neq 0$ is nonempty for any $j \geq 1$. In fact, the $\mathfrak{A}_j$ form a projective system (w.r.t.\ sending $A \subseteq G/N_{j+1}$ to its image in $G/N_j$) of finite sets. Hence their projective limit is nonempty. For any element $(A_j)_j \in \varprojlim \mathfrak{A}_j$ its projective limit $A := \varprojlim A_j$ is a closed abelian subgroup in $G$ such that $c|A \neq 0$. This shows that $c$ does not lie in the right hand kernel.
\end{proof}

\begin{corollary}\label{profinite}
We have
\begin{equation*}
     SK_1(\Lambda(G))^\vee = \ker\big(
     H^2(G,\mathbb{Q}_p/\mathbb{Z}_p)
     \xrightarrow{\; {\rm res} \;} \prod_{A \subseteq G}
     H^2(A,\mathbb{Q}_p/\mathbb{Z}_p) \big) \ .
\end{equation*}
\end{corollary}

In order to use this criterion we consider the connecting homomorphism
\begin{equation*}
    \delta : H^1(G, \mathbb{Q}_p/\mathbb{Z}_p) \longrightarrow H^2(G,\mathbb{F}_p)
\end{equation*}
in the cohomology sequence for multiplication by $p$ on $\mathbb{Q}_p/\mathbb{Z}_p$.

\begin{lemma}\label{connecting}
Suppose that $G$ is torsion free; then $SK_1(\Lambda(G))$ vanishes if and only if the sequence
\begin{equation*}
    0 \longrightarrow H^1(G, \mathbb{Q}_p/\mathbb{Z}_p)/p \xrightarrow{\; \delta \;} H^2(G,\mathbb{F}_p) \xrightarrow{\; {\rm res} \;}  \prod_{A \subseteq G} H^2(A,\mathbb{F}_p)
\end{equation*}
is exact.
\end{lemma}
\begin{proof}
Any closed abelian subgroup $A \subseteq G$ is of the form $A \cong \mathbb{Z}_p^s$ for some $s \in \mathbb{N}$. It follows that $H^1(A, \mathbb{Q}_p/\mathbb{Z}_p) \cong (\mathbb{Q}_p/\mathbb{Z}_p)^s$ is divisible and further that $H^2(A,\mathbb{F}_p) = H^2(A, \mathbb{Q}_p/\mathbb{Z}_p)[p]$. We therefore obtain the exact commutative diagram:
\begin{equation*}
    \xymatrix{
      0 \ar[d]  &  \\
      H^1(G, \mathbb{Q}_p/\mathbb{Z}_p)/p \ar[d]_{\delta} & \\
      H^2(G,\mathbb{F}_p) \ar[d] \ar[r] & \prod_A H^2(A,\mathbb{F}_p) \ar[d]^{\cong}  \\
      H^2(G, \mathbb{Q}_p/\mathbb{Z}_p)[p] \ar[d] \ar[r] & \prod_A H^2(A, \mathbb{Q}_p/\mathbb{Z}_p)[p] \\
      0  &    }
\end{equation*}
By Cor.\ \ref{profinite} the group $SK_1(\Lambda(G))$ vanishes if and only if the lower horizontal map in this diagram is injective.
\end{proof}

\section{Finiteness}

We abbreviate $\mathfrak{g} := \Lie_{\mathbb{Q}_p}(G)$. Let $Z$ denote the center of $G$ and $\mathfrak{z} := \Lie_{\mathbb{Q}_p}(Z)$ its Lie algebra. According to \cite{B-LL} III\S9 Ex.\ 4 we have
\begin{equation*}
    \mathfrak{z} = \{\mathfrak{x} \in \mathfrak{g} : \mathrm{Ad}(g)(\mathfrak{x}) = \mathfrak{x} \ \textrm{for any $g \in G$}\}.
\end{equation*}
\cite{B-LL} III\S9.3 Prop.\ 7 then implies that $\mathfrak{z}$ is contained in the center of $\mathfrak{g}$.

\begin{theorem}\label{finiteness}
Suppose that $\mathfrak{z}$ is the center of $\mathfrak{g}$ and that the Lie algebra $\mathfrak{g}/\mathfrak{z}$ of $G/Z$ is semisimple; then $SK_1(\Lambda(G))$ is finite.
\end{theorem}
\begin{proof}
In a first step we assume that $\mathfrak{g}$ is semisimple. Then
\begin{equation*}
    H^1(\mathfrak{g},\mathbb{Q}_p) = H^2(\mathfrak{g},\mathbb{Q}_p) = 0
\end{equation*}
by Whitehead's lemmata (cf.\ \cite{Wei} 7.8.10 and 7.8.12). The Lazard isomorphism for continuous group cohomology (\cite{Laz} V.2.4.10) says that
\begin{equation*}
    H^\ast_c(G',\mathbb{Q}_p) \cong H^\ast(\mathfrak{g},\mathbb{Q}_p)
\end{equation*}
for any sufficiently small open subgroup $G' \subseteq G$. Furthermore, we have (cf.\ \cite{NSW} 2.7.12)
\begin{equation*}
    \dim_{\mathbb{Q}_p} H^\ast_c(G',\mathbb{Q}_p) = \mathrm{corank}_{\mathbb{Z}_p} H^\ast(G',\mathbb{Q}_p/\mathbb{Z}_p) \ .
\end{equation*}
Altogether it follows that the groups $H^1(G',\mathbb{Q}_p/\mathbb{Z}_p)$ and $H^2(G',\mathbb{Q}_p/\mathbb{Z}_p)$ are finite. A straightforward argument with the Hochschild-Serre spectral sequence then implies that also the groups $H^1(G,\mathbb{Q}_p/\mathbb{Z}_p)$ and $H^2(G,\mathbb{Q}_p/\mathbb{Z}_p)$ are finite. Under the assumption that $\mathfrak{g}$ is semisimple we therefore have established the stronger assertion that already the middle term in the identity of Cor.\  \ref{profinite} is finite.

Now let $\mathfrak{g}$ be arbitrary as in the assertion. By our first step we know that the groups $H^1(G/Z,\mathbb{Q}_p/\mathbb{Z}_p)$ and $H^2(G/Z,\mathbb{Q}_p/\mathbb{Z}_p)$ are finite. The Hochschild-Serre spectral sequence gives rise to an embedding
\begin{equation*}
    0 \longrightarrow \frac{\ker(H^2(G,\mathbb{Q}_p/\mathbb{Z}_p) \rightarrow H^2(Z,\mathbb{Q}_p/\mathbb{Z}_p))} {\im(H^2(G/Z,\mathbb{Q}_p/\mathbb{Z}_p) \rightarrow H^2(G,\mathbb{Q}_p/\mathbb{Z}_p))} \longrightarrow H^1(G/Z,H^1(Z,\mathbb{Q}_p/\mathbb{Z}_p))
\end{equation*}
The denominator of the left term is finite. Since $Z$ is a finitely generated $\mathbb{Z}_p$-module its cohomology $H^1(Z,\mathbb{Q}_p/\mathbb{Z}_p)$ is a cofinitely generated abelian group which, moreover, is trivial as a $G/Z$-module. Hence the right hand term is finite as well. We conclude that $M := \ker(H^2(G,\mathbb{Q}_p/\mathbb{Z}_p) \rightarrow H^2(Z,\mathbb{Q}_p/\mathbb{Z}_p))$ is finite. Using Cor.\ \ref{profinite} we now consider the exact diagram:
\begin{equation*}
    \xymatrix{
       & & 0 \ar[d] &  \\
       & & SK_1(\Lambda(G))^\vee \ar[d] & & \\
        0 \ar[r] & M  \ar[r] & H^2(G, \mathbb{Q}_p/\mathbb{Z}_p) \ar[d] \ar[r] & H^2(Z, \mathbb{Q}_p/\mathbb{Z}_p)  \\
       & & \prod_A H^2(A, \mathbb{Q}_p/\mathbb{Z}_p) \ar@{-->}[ur] & \\
      }
\end{equation*}
Clearly there is the oblique arrow which makes the diagram commutative. The diagram therefore induces a surjection
\begin{equation*}
    \Hom(M , \mathbb{Q}_p/\mathbb{Z}_p) \longrightarrow SK_1(\Lambda(G)) \longrightarrow 0
\end{equation*}
and shows in this way that $SK_1(\Lambda(G))$ is finite.
\end{proof}

In the light of Cor.\ \ref{profinite} the following fact is of some interest. We assume that
\begin{equation*}
    \textrm{$G$ is saturated rationally equi-$p$-valued}
\end{equation*} in the sense of \cite{Laz} . In this situation Lazard has constructed in \cite{Laz} IV.3.3.6 a $\mathbb{Z}_p$-Lie subalgebra $\mathcal{L} := \mathcal{L}(G) \subseteq \mathfrak{g}$ which is a $\mathbb{Z}_p$-lattice. Furthermore, he also has shown in \cite{Laz} V.2.5.7.1, that the cohomology algebra
\begin{equation*}
    H^\ast(G, \mathbb{F}_p) = \bigwedge H^1(G, \mathbb{F}_p)
\end{equation*}
is the exterior algebra over $H^1(G, \mathbb{F}_p)$. Both of these results will be of fundamental importance for everything which follows in this paper.

\begin{proposition}\label{integral}
Suppose that $G$ satisfies the above condition and that
$\mathfrak{g}$ is semisimple, and put $d := \dim G =
\dim_{\mathbb{Q}_p} \mathfrak{g}$; we then have $\dim_{\mathbb{F}_p} H^2(G, \mathbb{Q}_p/\mathbb{Z}_p)[p] = \binom{d}{2} - d$.
\end{proposition}
\begin{proof}
Under our assumption on $G$ the Lazard isomorphism was sharpened in
\cite{HKN}Thm.\ 3.1.1 to an integral isomorphism
\begin{equation*}
    H^\ast_c(G,\mathbb{Z}_p) \otimes_{\mathbb{Z}_p} \mathcal{O} \cong H^\ast(\mathcal{L},\mathbb{Z}_p) \otimes_{\mathbb{Z}_p} \mathcal{O}
\end{equation*}
where $\mathcal{O}$ is the ring of integers in an appropriate finite extension $K$ of $\mathbb{Q}_p$.
Using Whitehead's lemma we in particular obtain
\begin{equation*}
    H^i_c(G,\mathbb{Z}_p) \otimes_{\mathbb{Z}_p}
    K = H^i(\mathcal{L},\mathbb{Z}_p)
    \otimes_{\mathbb{Z}_p} K = H^i(\mathfrak{g},K) =
    0
\end{equation*}
for $i = 1, 2$.
It follows that $H^2_c(G,\mathbb{Z}_p)$ must be finite and that, since
$H^1_c(G,\mathbb{Z}_p)$ by construction is torsion free, we even
have $H^1_c(G,\mathbb{Z}_p) = 0$. The cohomology sequence for
multiplication by $p$ on $\mathbb{Z}_p$ therefore gives
\begin{equation*}
    H^1(G,\mathbb{F}_p) = H^2_c(G,\mathbb{Z}_p)[p] \cong H^2_c(G,\mathbb{Z}_p)/p
\end{equation*}
as well as the exact sequence
\begin{equation*}
    0 \longrightarrow H^2_c(G,\mathbb{Z}_p)/p \longrightarrow
    H^2(G,\mathbb{F}_p) \longrightarrow H^3_c(G,\mathbb{Z}_p)[p] = H^2(G,
    \mathbb{Q}_p/\mathbb{Z}_p)[p] \longrightarrow 0 \ .
\end{equation*}
It follows that
\begin{equation*}
    \dim_{\mathbb{F}_p} H^2(G, \mathbb{Q}_p/\mathbb{Z}_p)[p] =
    \dim_{\mathbb{F}_p} H^2(G,
    \mathbb{F}_p) - \dim_{\mathbb{F}_p} H^1(G, \mathbb{F}_p) \ .
\end{equation*}
By our assumption on the group $G$ the $p$-th powers $G^p$ form a
normal subgroup such that $G/G^p$ is an $\mathbb{F}_p$-vector space
of dimension $d$. Hence $\dim_{\mathbb{F}_p} H^1(G, \mathbb{F}_p) =
d$. But $H^2(G, \mathbb{F}_p) = \bigwedge^2 H^1(G, \mathbb{F}_p)$. Hence $\dim_{\mathbb{F}_p} H^2(G, \mathbb{F}_p) = \binom{d}{2}$.
\end{proof}

\begin{corollary}\label{SL_2}
For any open subgroup $G \subseteq SL_2(\mathbb{Z}_p)$ which satisfies the above condition we have $SK_1(\Lambda(G)) = 0$.
\end{corollary}
\begin{proof}
In this case we have $d=3$ so that $H^2(G,
\mathbb{Q}_p/\mathbb{Z}_p)=0$ by Prop.\ \ref{integral}.
\end{proof}

From now on we always will assume that our group $G$ is uniform in the sense of \cite{DDMS}. In Lazard's language this is equivalent to $G$ being saturated integrally $p$-valued.

\section{The connecting homomorphism}

Our goal in this section is to
explicitly compute the connecting homomorphism $\delta$ in Lemma \ref{connecting} (under our standing assumption that $G$ is uniform, of course).

As recalled before the proof of Prop.\ \ref{integral} we have
\begin{equation*}
    H^2(G,\mathbb{F}_p) = \bigwedge^2 H^1(G,\mathbb{F}_p) = \bigwedge^2 \Hom(G/G^p, \mathbb{F}_p) =\Hom(\bigwedge^2 (G/G^p),\mathbb{F}_p) \ .
\end{equation*}
we also have noted already that in a uniform group the subset $G^p := \{g^p : g \in G\}$ in fact is an open normal subgroup. The quotient $G/G^p$, of course, is an $\mathbb{F}_p$-vector space.
For simpler reasons (cf.\ \cite{Ser} p.\ 117) the corresponding fact
\begin{equation*}
    H^2(A,\mathbb{F}_p) = \bigwedge^2 H^1(A,\mathbb{F}_p) = \bigwedge^2 \Hom(A/A^p, \mathbb{F}_p) = \Hom(\bigwedge^2 (A/A^p),\mathbb{F}_p)
\end{equation*}
holds true for any $A \cong \mathbb{Z}_p^s$. In the following we write $V := G/G^p$ and we introduce in $\bigwedge^2 V$ the $\mathbb{F}_p$-vector subspace $\bigwedge V$ generated by all elements of the form $aG^p \wedge bG^p$ where $a$ and $b$ are contained in the same  closed abelian subgroup of $G$ (so that $[a,b] = 1$). Then
\begin{equation*}
    \ker\big( H^2(G,\mathbb{F}_p) \longrightarrow \prod_A H^2(A,\mathbb{F}_p) \big) = \Hom(\bigwedge^2 V \big/ \bigwedge V, \mathbb{F}_p) \ .
\end{equation*}
The connecting homomorphism $\delta$, by Pontrjagin duality, gives rise to a homomorphism
\begin{equation*}
    \delta^\vee : \bigwedge^2 V \longrightarrow G/[G,G] =: G^{ab} \ .
\end{equation*}
We obtain the complex
\begin{equation}\label{f:complex}
    0 \longrightarrow \bigwedge V \xrightarrow{\; \subseteq \;} \bigwedge^2 V \xrightarrow{\; \delta^\vee \;} G^{ab}[p] \longrightarrow 0 \ ,
\end{equation}
which is exact except possibly in the middle, together with the following reformulation of Lemma \ref{connecting}.

\begin{lemma}\label{exactness}
We have $SK_1(\Lambda(G)) = 0$ if and only if the complex \eqref{f:complex} is exact.
\end{lemma}

On the other hand it is easy to see that
\begin{align*}
    \partial \qquad : V \wedge V & \longrightarrow (G^p/[G^p,G])[p] \\
    gG^p \wedge hG^p & \longmapsto [g,h] \bmod [G^p,G]
\end{align*}
is a well defined $\mathbb{F}_p$-linear map. In oder to compare these two maps we claim that
\begin{align*}
    \varphi : \qquad G^{ab} & \longrightarrow G^p/[G^p,G] \\
               g[G,G] & \longmapsto g^p [G^p,G]
\end{align*}
is well defined and multiplicative. This will be a consequence of the Hall-Petrescu formula (cf.\ \cite{B-LL} II \S5 Ex.\ 9). First of all, 9.d) implies that
\begin{equation*}
    [g,h]^p \equiv [g^p,h] \equiv 1 \bmod [G^p,G]
\end{equation*}
for any $g,h \in G$. Using 9.a) we then conclude that
\begin{align*}
    (gh)^p & \equiv g^p h^p (h^{-1} [h^{-1},g^{-1}]h)^{\binom{p}{2}} \bmod [G^p,G] \\
    & \equiv g^p h^p (h^{-1} [h^{-1},g^{-1}]^p h)^{\frac{p-1}{2}} \bmod [G^p,G] \\
    & \equiv g^p h^p \bmod [G^p,G]
\end{align*}
for any $g,h \in G$. In particular, if $h = [h_1,h_2]$ is any commutator, then
\begin{equation*}
    (gh)^p \equiv g^p [h_1,h_2]^p \equiv g^p \bmod [G^p,G] \ .
\end{equation*}

\begin{lemma}\label{phi-iso}
$\varphi$ is an isomorphism.
\end{lemma}
\begin{proof}
The surjectivity is obvious. For the injectivity we first of all note that according to \cite{GS} Thm.\ 3.4 our group $G$ is a torsion free
PF-group. By loc.\ cit. Prop.\ 2.1(3)-(5) the commutator subgroup
$[G,G]$ therefore is a PF-embedded subgroup which satisfies $[G,G]^p
= [G^p,G] = [G^p,G]$. Let now $g[G,G]$ be any element in the kernel
of $\varphi$. This means that $g \in [G^p,G] = [G,G]^p$. Then loc.\
cit.\ Prop.\ 2.2 implies that $g \in [G,G]$.
\end{proof}

\begin{proposition}\label{partial}
$\varphi \circ \delta^\vee = \partial$.
\end{proposition}
\begin{proof}
We will in fact establish the dual identity $\delta \circ \varphi^\vee = \partial^\vee$.

\textit{Step 1:} We compute the connecting homomorphism $\delta : \Hom^{cont}(G, \mathbb{Q}_p/\mathbb{Z}_p) \longrightarrow H^2(G,\mathbb{F}_p)$, where here and in the following the superscript ``cont'' indicates the subgroup of continuous homomorphisms. Let $f \in \Hom^{cont}(G, \mathbb{Q}_p/\mathbb{Z}_p)$ and let $\tilde{f} : G \longrightarrow \mathbb{Q}_p/\mathbb{Z}_p$ be any continuous map (of sets) such that $p \tilde{f} = f$. Then $\delta(f) = [d_{\tilde{f}}]$ is the class of the 2-cocycle $d_{\tilde{f}}(g,h) := - \tilde{f}(gh) + \tilde{f}(g) + \tilde{f}(h)$ on $G$. In order to make a suitable choice for $\tilde{f}$ we fix an ordered basis $(g_1, \ldots,g_d)$ of $G$. For any $p$-adic integer $c \in \mathbb{Z}_p$ we let $0 \leq \bar{c} \leq p-1$ denote the unique integer such that $c \equiv \bar{c} \bmod p$. If $g = g_1^{c_1} \ldots g_d^{c_d} \in G$ is any element we find a unique element $\pi(g) \in G$ such that
\begin{equation*}
    g = g_1^{\bar{c}_1} \ldots g_d^{\bar{c}_d} \pi(g)^p
\end{equation*}
(cf.\ \cite{DDMS} \S4.2). We fix elements $f_1, \ldots, f_d \in \mathbb{Q}_p/\mathbb{Z}_p$ such that $pf_i = f(g_i)$, and we choose
\begin{equation*}
    \tilde{f}(g) := \sum_{i=1}^d \bar{c}_i f_i + f(\pi(g)) \ .
\end{equation*}
Let now $g = g_1^{a_1} \ldots g_d^{a_d}$ and $h = g_1^{b_1} \ldots g_d^{b_d}$ be any two elements in $G$. Then
\begin{equation*}
    \tilde{f}(g) = \sum_{i=1}^d \bar{a}_i f_i + f(\pi(g)) \qquad\textrm{and}\qquad \tilde{f}(h) = \sum_{i=1}^d \bar{b}_i f_i + f(\pi(h)) \ .
\end{equation*}
On the other hand setting
\begin{equation*}
    p_i :=
          \begin{cases}
             p & \textrm{if $\bar{a}_i + \bar{b}_i
                \geq p$} , \\
             0 & \textrm{otherwise}
          \end{cases}
\end{equation*}
we have
\begin{align*}
    gh & = g_1^{\bar{a}_1} \ldots g_d^{\bar{a}_d} \pi(g)^p g_1^{\bar{b}_1} \ldots g_d^{\bar{b}_d} \pi(h)^p \\
& \equiv g_1^{\bar{a}_1 + \bar{b}_1} \ldots g_d^{\bar{a}_d + \bar{b}_d} \pi(g)^p \pi(h)^p \prod_{i > j} [g_i^{-\bar{a}_i} , g_j^{-\bar{b}_j}] \bmod [G^p,G] \\
& \equiv g_1^{\bar{a}_1 + \bar{b}_1 - p_1} \ldots g_d^{\bar{a}_d + \bar{b}_d - p_d}
\pi(g)^p \pi(h)^p (\prod_{\bar{a}_j + \bar{b}_j
                \geq p} g_j^p) \prod_{i > j} [g_i , g_j]^{\bar{a}_i \bar{b}_j} \bmod [G^p,G] \ .
\end{align*}
Since $[G^p,G] = [G,G]^p$ this means that
\begin{equation*}
    gh = g_1^{\bar{a}_1 + \bar{b}_1 - p_1} \ldots g_d^{\bar{a}_d + \bar{b}_d - p_d}
\pi(g)^p \pi(h)^p (\prod_{\bar{a}_j + \bar{b}_j
                \geq p} g_j^p) \prod_{i > j} [g_i , g_j]^{\bar{a}_i \bar{b}_j} \gamma_1^p
\end{equation*}
for some $\gamma_1 \in [G,G]$. The element $\pi(g)^p \pi(h)^p (\prod_{\bar{a}_j + \bar{b}_j
                \geq p} g_j^p) \prod_{i > j} [g_i , g_j]^{\bar{a}_i \bar{b}_j} \gamma_1^p$ lies in $G^p$ and hence is equal to $y^p$ for some $y \in G$. Since then
\begin{equation*}
    \varphi \big(\pi(g) \pi(h) (\prod_{\bar{a}_j + \bar{b}_j \geq p} g_j) \prod_{i > j} ([g_i , g_j]^{\frac{1}{p}})^{\bar{a}_i \bar{b}_j} \gamma_1 \big) = \varphi(y)
\end{equation*}
Lemma \ref{phi-iso} implies that
\begin{equation*}
    \pi(g) \pi(h) (\prod_{\bar{a}_j + \bar{b}_j \geq p} g_j) \prod_{i > j} ([g_i , g_j]^{\frac{1}{p}})^{\bar{a}_i \bar{b}_j} \gamma_1 = y \gamma_2
\end{equation*}
for some $\gamma_2 \in [G,G]$. We obtain
\begin{equation*}
    \pi(g)^p \pi(h)^p (\prod_{\bar{a}_j + \bar{b}_j
                \geq p} g_j^p) \prod_{i > j} [g_i , g_j]^{\bar{a}_i \bar{b}_j} \gamma_1^p = \big( \pi(g) \pi(h) (\prod_{\bar{a}_j + \bar{b}_j \geq p} g_j) \prod_{i > j} ([g_i , g_j]^{\frac{1}{p}})^{\bar{a}_i \bar{b}_j} \gamma_3 \big)^p
\end{equation*}
with $\gamma_3 := \gamma_1 \gamma_2^{-1} \in [G,G]$ and hence
\begin{equation*}
    gh = g_1^{\bar{a}_1 + \bar{b}_1 - p_1} \ldots g_d^{\bar{a}_d + \bar{b}_d - p_d} \big( \pi(g) \pi(h) (\prod_{\bar{a}_j + \bar{b}_j \geq p} g_j) \prod_{i > j} ([g_i , g_j]^{\frac{1}{p}})^{\bar{a}_i \bar{b}_j} \gamma_3 \big)^p \ .
\end{equation*}
This shows that
\begin{align*}
    \tilde{f}(gh) & = \sum_i (\bar{a}_i + \bar{b}_i - p_i)f_i + f \big(\pi(g) \pi(h) (\prod_{\bar{a}_j + \bar{b}_j \geq p} g_j) \prod_{i > j} ([g_i , g_j]^{\frac{1}{p}})^{\bar{a}_i \bar{b}_j} \gamma_3 \big) \\
& = \sum_i (\bar{a}_i + \bar{b}_i - p_i)f_i + f(\pi(g)) + f(\pi(h)) + \sum_{\bar{a}_j + \bar{b}_j \geq p} f(g_j) + \sum_{i > j} \bar{a}_i \bar{b}_j f([g_i , g_j]^{\frac{1}{p}}) \\
& = \sum_i (\bar{a}_i + \bar{b}_i)f_i + f(\pi(g)) + f(\pi(h)) + \sum_{i > j} \bar{a}_i \bar{b}_j f([g_i , g_j]^{\frac{1}{p}})
\end{align*}
and consequently that
\begin{equation*}
    d_{\tilde{f}}(g,h) = - \sum_{i > j} \bar{a}_i \bar{b}_j f([g_i , g_j]^{\frac{1}{p}}) \ .
\end{equation*}

\textit{Step 2:} Let $f \in (G^p/[G^p,G])^\vee = \Hom_G^{cont}(G^p, \mathbb{Q}_p/\mathbb{Z}_p)$ be a continuous $G$-equivariant (for the adjoint action of $G$ on $G^p$) homomorphism. We obtain from Step 1 that $\delta \circ \varphi^\vee (f)$ is the class of the 2-cocycle
\begin{equation}\label{f:cocycle}
    (g,h) = (g_1^{a_1} \ldots g_d^{a_d}, g_1^{b_1} \ldots g_d^{b_d}) \longmapsto - \sum_{i > j} \bar{a}_i \bar{b}_j f([g_i , g_j]) \ .
\end{equation}

\textit{Step 3:} We compute the map
\begin{equation*}
    \partial^\vee : \Hom_G^{cont}(G^p, \mathbb{Q}_p/\mathbb{Z}_p) \longrightarrow (\bigwedge^2 V)^\vee = \bigwedge^2 V^\vee = \bigwedge^2 H^1(G,\mathbb{F}_p) = H^2(G,\mathbb{F}_p) \ .
\end{equation*}
We emphasize that the identification $\bigwedge^2 H^1(G,\mathbb{F}_p) = H^2(G,\mathbb{F}_p)$ is
given by the cup-product. The elements $e_1 := g_1G^p, \ldots, e_d := g_dG^p$ form an $\mathbb{F}_p$-basis of $V$; let $e_1^\vee, \ldots, e_d^\vee$ denote the dual
basis. Let $f \in \Hom_G^{cont}(G^p, \mathbb{Q}_p/\mathbb{Z}_p)$. Then $\partial^\vee(f)$ is given,
in $(\bigwedge^2 V)^\vee$, by $e_i \wedge e_j \longmapsto f([g_i,g_j])$. We observe that under the
identification $(\bigwedge^2 V)^\vee = \bigwedge^2 V^\vee$ the basis dual to $e_i \wedge e_j$ (for
$i > j$) corresponds to $e_i^\vee \wedge e_j^\vee$. Hence
\begin{equation*}
    \partial^\vee(f) = \sum_{i > j} f([g_i,g_j]) e_i^\vee \wedge e_j^\vee
\end{equation*}
in $\bigwedge^2 V^\vee$. Finally applying the cup-product shows that in $H^2(G,\mathbb{F}_p)$ the class $\partial^\vee(f)$ is given by the 2-cocycle
\begin{align*}
    (g,h) = (g_1^{a_1} \ldots g_d^{a_d}, g_1^{b_1} \ldots g_d^{b_d}) \longmapsto - \sum_{i > j} f([g_i,g_j]) e_i^\vee(gG^p)e_j^\vee(hG^p) = - \sum_{i > j} \bar{a}_i \bar{b}_j f([g_i , g_j]) \ .
\end{align*}
This coincides with \eqref{f:cocycle} and therefore establishes our assertion.
\end{proof}

\begin{corollary}\label{S-criterion}
We have $SK_1(\Lambda(G)) = 0$ if and only if $\bigwedge V = \ker \partial$.
\end{corollary}

\section{A Lie criterion}

We recall that $\mathcal{L} := \mathcal{L}(G)$ denotes the $\mathbb{Z}_p$-Lie algebra of the uniform group $G$ in the sense of Lazard (cf.\ \cite{DDMS} \S4.5). As sets $G$ and $\mathcal{L}$ coincide. To avoid confusion we write, for any two elements $g,h \in G = \mathcal{L}$, in the following $[g,h]_G$ and $[g,h]_L$ for the commutators in the group $G$ and in the Lie algebra $\mathcal{L}$, respectively. The Lie algebra structure is given in terms of the group structure as follows:
\begin{itemize}
  \item[--] $g + h = \lim_{n \rightarrow \infty} (g^{p^n} h^{p^n})^{p^{-n}}$;
  \item[--] in particular, $xg = g^x$ for any $x \in \mathbb{Z}_p$;
  \item[--] $[g,h]_L = \lim_{n \rightarrow \infty} [g^{p^n}, h^{p^n}]_G^{p^{-2n}}$.
\end{itemize}
According to \cite{GS} Prop.\ 2.1(5) we have
\begin{equation}\label{f:p.e.}
    [G^p,G]_G = [G,G]_G^p = [G,G^p]_G
\end{equation}
and hence $[G,[G,G]_G]_G \subseteq [G,G^p]_G \subseteq [G,G]_G^p$, which says that $[G,G]_G$ is powerfully embedded in $G$ and therefore, in particular, is uniform. By \cite{GS} Prop.\ 2.1(3) and Thm.\ B(1) we have
\begin{equation*}
    \mathcal{L}([G,G]_G) = [\mathcal{L},\mathcal{L}]_L \qquad\textrm{and} \qquad \mathcal{L}([G^p,G]_G) = p[\mathcal{L},\mathcal{L}]_L \ .
\end{equation*}

\begin{remark}\label{add}
\begin{itemize}
  \item[i.] The map $G/G^p \xrightarrow{\; \id \;} \mathcal{L}/p \mathcal{L}$ is an isomorphism of $\mathbb{F}_p$-vector spaces.
  \item[ii.] The map $G^p/[G^p,G]_G \xrightarrow{\; \id \;} p\mathcal{L}/ [p\mathcal{L},\mathcal{L}]_L$ is an isomorphism of $\mathbb{Z}_p$-modules.
\end{itemize}
\end{remark}
\begin{proof}
i. See \cite{DDMS} Cor.\ 4.15. ii. Using Lemma \ref{phi-iso} it suffices to show that the map
$G/[G,G]_G \xrightarrow{\; \id \;} \mathcal{L}/ [\mathcal{L},\mathcal{L}]_L$ is an isomorphism of
$\mathbb{Z}_p$-modules. This means that we have to check that
\begin{equation*}
    gh \equiv g + h \bmod [G,G]_G
\end{equation*}
holds true for any $g,h \in G$. This will follow from part ii. of
the following lemma.
\end{proof}

\begin{lemma}
\begin{enumerate}
\item[i.]  $(gh)^{p^n}\equiv g^{p^n}h^{p^n} \bmod  [G,G]_G^{p^n}$ for all $n\geq 0.$
\item[ii.]$gh\equiv(g^{p^n}h^{p^n})^{p^{-n}}\bmod [G,G]_G$ for all $n\geq 0.$
\end{enumerate}
\end{lemma}

\begin{proof}
i. According to the Hall-Petrescu formula \cite{DDMS} App.\ A, for
any group $F$ and $x,y\in F$ one has
\begin{equation*}
    (xy)^p=x^py^pc_2^{\binom{p }{2}}
    \cdots c_j^{\binom{p}{j}}\cdots c_{p-1}^pc_p
\end{equation*}
for some $c_j\in \gamma_j(F),$ where
$\gamma_{j+1}(F):=[\gamma_j(F),F]$ with $\gamma_1(F) := F$ denotes the lower central series. We apply this formula to the group $G^{p^i}$,
for any $i\geq 0$. Using
\begin{equation*}
    [G^{p^i},G^{p^j}]_G=[G,G]_G^{p^{i+j}}\subseteq G^{p^{i+j +1}}
\end{equation*}
by \cite{GS} Prop.\ 2.1(3)-(5) as well as (by induction)
\begin{equation*}
    \gamma_j(G^{p^i})\subseteq [G,G]_G^{p^{ij+j-2}}
    \qquad\textrm{for $i \geq 0$, $j\geq 2$}
\end{equation*}
we obtain
\begin{equation}\label{f:2i}
    (gh)^p\equiv g^ph^p\bmod [G,G]_G^{p^{2i}} \qquad\textrm{for
    $g,h\in G^{p^i}$ and $i\geq 0$.}
\end{equation}
We now prove the assertion via induction. The case $n=0$ is trivial
and the case $n=1$ was shown before Lemma \ref{phi-iso}. For $n \geq 1$ we
have, by the induction hypothesis, that
\begin{align*}
(gh)^{p^{n+1}} & = (g^{p^n}h^{p^n}y^{p^n})^p\\
 & \equiv (g^{p^n}h^{p^n})^py^{p^{n+1}}\bmod [G,G]_G^{p^{2n}}\\
 & \equiv g^{p^{n+1}}h^{p^{n+1}}  \bmod [G,G]_G^{p^{n+1}},
\end{align*}
for some $y\in [G,G]_G$ where for the two congruences we used \eqref{f:2i}.

ii. By \cite{GS} Prop.\ 2.2 the element
$(gh)^{-1}(g^{p^n}h^{p^n})^{p^{-n}}$ belongs to $[G,G]_G$ if and
only if $\big((gh)^{-1}(g^{p^n}h^{p^n})^{p^{-n}}\big)^{p^n}$ belongs
to $[G,G]_G^{p^n}.$ But by i. we do have
\begin{equation*}
\big((gh)^{-1}(g^{p^n}h^{p^n})^{p^{-n}}\big)^{p^n} \equiv
(gh)^{-p^n} g^{p^n}h^{p^n} \equiv 1 \bmod [G,G]_G^{p^n} \ .
\end{equation*}
\end{proof}

This means that source and target of the two maps
\begin{equation}\label{f:two-maps}
    G/G^p \wedge G/G^p \xrightarrow{\; [\; ,\; ]_G \;} G^p/[G^p,G] \qquad\textrm{and}\qquad \mathcal{L}/p \mathcal{L} \wedge \mathcal{L}/p \mathcal{L} \xrightarrow{\; [\; ,\; ]_L \;} p\mathcal{L}/ p[\mathcal{L},\mathcal{L}]_L
\end{equation}
coincide. The next lemma implies that, in fact, the two maps coincide.

\begin{lemma}\label{comm-congr}
$[g,h]_G \equiv [g^{p^n},h^{p^n}]_G^{p^{-2n}} \bmod [G^p,G]_G$ for any $n \geq 0$ and any $g,h \in G$.
\end{lemma}
\begin{proof}
The Hall-Petrescu identity says (cf.\ \cite{B-LL} II \S5 Ex.\ 9d)) that
\begin{equation*}
    [g^{p^n},h]_G = [g,h]_G^{p^n} v_3(g,h)^{\binom{p^n}{2}} \ldots v_{p^n}(g,h)^{p^n} v_{p^n +1}(g,h)
\end{equation*}
where the $v_j(g,h)$ are universal products of iterated commutators in $g$ and $h$ of length $j$. In particular,
\begin{equation*}
    [g^{p^n},h^{p^n}]_G = [g,h^{p^n}]_G^{p^n} v_3(g,h^{p^n})^{\binom{p^n}{2}} \ldots v_{p^n}(g,h^{p^n})^{p^n} v_{p^n +1}(g,h^{p^n}) \ .
\end{equation*}
From $[G,G]_G \subseteq G^p$ and formula \eqref{f:p.e.} we conclude
inductively that
\begin{equation*}
    v_j(g,h) \in [G,G]_G^{p^{j-2}} \qquad\textrm{and}\qquad v_j(g,h^{p^n}) \in [G,G]_G^{p^{n+j-2}} \ .
\end{equation*}
It is well known that the $p$-adic valuation of $(j-1)!$ satisfies
\begin{equation*}
    \mathrm{v}_p((j-1)!) \leq \frac{j-2}{p-1} \ .
\end{equation*}
Since $p > 2$ we obtain $\mathrm{v}_p((j-1)!) \leq j-3$ for $j \geq 3$ and hence
\begin{equation*}
    v_j(g,h)^{\binom{p^n}{j-1}} \in [G,G]_G^{p^{n+1}} \qquad\textrm{and}\qquad v_j(g,h^{p^n})^{\binom{p^n}{j-1}} \in [G,G]_G^{p^{2n+1}} \ .
\end{equation*}
It follows that
\begin{equation*}
    [g^{p^n},h]_G \in [g,h]_G^{p^n} [G,G]_G^{p^{n+1}} \qquad\textrm{and}\qquad [g^{p^n},h^{p^n}]_G \in [g,h^{p^n}]_G^{p^n} [G,G]_G^{p^{2n+1}} \ .
\end{equation*}
By passing to inverses we also have $[g,h^{p^n}]_G \in [g,h]_G^{p^n} [G,G]_G^{p^{n+1}}$. Inserting this into the right hand formula gives
\begin{equation*}
    [g^{p^n},h^{p^n}]_G \in ([g,h]_G^{p^n} [G,G]_G^{p^{n+1}})^{p^n} [G,G]_G^{p^{2n+1}} \ .
\end{equation*}
But $([g,h]_G^{p^n} [G,G]_G^{p^{n+1}})^{p^n} \subseteq [g,h]_G^{p^{2n}} [G,G]_G^{p^{2n+1}}$, since $[G,G]_G$ is uniform, and hence
\begin{equation*}
    [g^{p^n},h^{p^n}]_G \equiv [g,h]_G^{p^{2n}} \bmod [G,G]_G^{p^{2n+1}} \ .
\end{equation*}
On the other hand, \eqref{f:p.e.} also implies that $[g^{p^n},h^{p^n}]_G \in [G,G]_G^{p^{2n}}$ which means, again since $[G,G]_G$ is uniform, that $[g^{p^n},h^{p^n}]_G = y^{p^{2n}}$ for some $y \in [G,G]_G$. We see that
\begin{equation*}
    y^{p^{2n}} \equiv [g,h]_G^{p^{2n}} \bmod [G,G]_G^{p^{2n+1}} \ ,
\end{equation*}
which implies, a third time since $[G,G]_G$ is uniform, that
\begin{equation*}
    y \equiv [g,h]_G \bmod [G,G]_G^p \ .
\end{equation*}
\end{proof}

In the previous section we had introduced the $\mathbb{F}_p$-vector subspace $\bigwedge (G/G^p) \subseteq \bigwedge^2 (G/G^p)$ which is generated by all elements of the form $gG^p \wedge hG^p$ where $[g,h]_G = 1$.

\begin{lemma}\label{abelian}
For any $g,h \in G$ we have $[g,h]_G = 1$ if and only if $[g,h]_L = 0$.
\end{lemma}
\begin{proof}
If $[g,h]_G = 1$ then $[g,h]_L = 0$ is immediate from the definition of the Lie bracket. Let us suppose that $[g,h]_L = 0$. Then $\mathcal{A} := \{ x \in \mathcal{L} : x^m \in \mathbb{Z}_p g + \mathbb{Z}_p h\ \textrm{for some $m \geq 1$} \}$ is an abelian Lie subalgebra of $\mathcal{L}$ in the sense of \cite{DDMS} Scholium to Thm.\ 9.10. Hence loc.\ cit.\ says that $\mathcal{A}$ at the same time is a uniform subgroup of $G$ which, moreover, by \cite{DDMS} Cor.\ 7.16(i) is commutative. We conclude that $[g,h]_G = 1$ holds true.
\end{proof}

We now introduce in $\bigwedge^2 \mathcal{L}$ the $\mathbb{Z}_p$-submodule $\bigwedge \mathcal{L}$ generated by all elements of the form $x \wedge y$ where $[x,y]_L = 0$. The above lemma implies that
\begin{equation}\label{f:transfer}
    \bigwedge (G/G^p) = (\bigwedge \mathcal{L} + p \bigwedge^2 \mathcal{L}) / p \bigwedge^2 \mathcal{L} \ .
\end{equation}
We therefore may reformulate Cor.\ \ref{S-criterion}  in the following form thereby obtaining our second main result.

\begin{theorem}\label{Lie-criterion1}
For uniform $G$ we have $SK_1(\Lambda(G)) = 0$ if and only if
\begin{equation*}
    \ker( \bigwedge^2 \mathcal{L} \xrightarrow{\; [\; ,\; ]_L \;} \mathcal{L}) \subseteq \bigwedge \mathcal{L} \ .
\end{equation*}
\end{theorem}
\begin{proof}
The identity \eqref{f:transfer} implies that $SK_1(\Lambda(G)) = 0$ if and only if
\begin{equation*}
    \bigwedge \mathcal{L} + p \bigwedge^2 \mathcal{L} = \ker( \bigwedge^2 \mathcal{L} \xrightarrow{\; [\; ,\; ]_L \;} p\mathcal{L}/ p[\mathcal{L},\mathcal{L}]_L) \ .
\end{equation*}
But the map $[\; ,\; ]_L : \bigwedge^2 \mathcal{L} \longrightarrow [\mathcal{L},\mathcal{L}]_L$ is surjective. Hence
\begin{equation*}
    \ker( \bigwedge^2 \mathcal{L} \xrightarrow{\; [\; ,\; ]_L \;} p\mathcal{L}/ p[\mathcal{L},\mathcal{L}]_L) = \ker( \bigwedge^2 \mathcal{L} \xrightarrow{\; [\; ,\; ]_L \;} \mathcal{L}) + p \bigwedge^2 \mathcal{L} \ .
\end{equation*}
It follows that $SK_1(\Lambda(G)) = 0$ if and only if
\begin{equation*}
    \bigwedge \mathcal{L} + p \bigwedge^2 \mathcal{L} = \ker( \bigwedge^2 \mathcal{L} \xrightarrow{\; [\; ,\; ]_L \;} \mathcal{L}) + p \bigwedge^2 \mathcal{L} \ .
\end{equation*}
For trivial reasons we have $\bigwedge \mathcal{L} \subseteq \ker( \bigwedge^2 \mathcal{L} \xrightarrow{\; [\; ,\; ]_L \;} \mathcal{L})$ so that the equivalent condition reduces to
\begin{equation*}
    \ker( \bigwedge^2 \mathcal{L} \xrightarrow{\; [\; ,\; ]_L \;} \mathcal{L}) \subseteq \bigwedge \mathcal{L} + p \bigwedge^2 \mathcal{L}  \ .
\end{equation*}
But this latter inclusion implies, by an easy inductive argument, that
\begin{equation*}
    \ker( \bigwedge^2 \mathcal{L} \xrightarrow{\; [\; ,\; ]_L \;} \mathcal{L}) \subseteq \bigwedge \mathcal{L} + p^j \bigwedge^2 \mathcal{L} \qquad\textrm{for any $j \geq 1$}
\end{equation*}
and hence that
\begin{equation*}
    \ker( \bigwedge^2 \mathcal{L} \xrightarrow{\; [\; ,\; ]_L \;} \mathcal{L}) \subseteq \bigwedge \mathcal{L} + \bigcap_j p^j \bigwedge^2 \mathcal{L} = \bigwedge \mathcal{L} \ .
\end{equation*}
The reverse direction is trivial.
\end{proof}

We observe that the criterion in Thm.\  \ref{Lie-criterion1} behaves well under products.

\begin{remark}
Let $\mathcal{L}_i$  be $\mathbb{Z}_p$-Lie algebras which satisfy the condition $\bigwedge \mathcal{L}_i=\ker [\;,\;]_i,$ and let $\mathcal{L}_0:=\mathcal{L}_1\oplus\mathcal{L}_2$ be the product
Lie algebra with bracket $[x_1+x_2,y_1+y_2]_0:=[x_1,y_1]_1 + [x_2,y_2]_2$ for $x_i,y_i\in\mathcal{L}_i$.
Then also for $\mathcal{L}_0$ we have
\begin{equation*}
    \bigwedge \mathcal{L}_0 =\ker
[\;,\;]_0 \ .
\end{equation*}
\end{remark}
\begin{proof}
The following identities are easily checked and imply the assertion:
\begin{gather*}
    \ker [\;,\;]_0 = \ker
[\;,\;]_1 + \mathcal{L}_1 \otimes_{\mathbb{Z}_p} \mathcal{L}_2  +\ker [\;,\;]_2 \ , \\
\bigwedge \mathcal{L}_0=\bigwedge\mathcal{L}_1 + \mathcal{L}_1 \otimes_{\mathbb{Z}_p} \mathcal{L}_2  +\bigwedge\mathcal{L}_2 \ .
\end{gather*}
\end{proof}

\section{A counterexample}\label{s:counterex}

In the following we will modify the Example 8.11 in \cite{Oli} in order to show that the criterion in Thm.\ \ref{Lie-criterion1} is not always satisfied.

Let $R$ be any commutative ring in which 2 is invertible. First we introduce the free $R$-module $V := R^4$ with its standard basis $e_1,\ldots, e_4$ and $W := \bigwedge^2 V/ R(e_1 \wedge e_2 + e_3 \wedge e_4)$ (which has rank 5) together with the linear map
\begin{equation*}
    \partial : \bigwedge^2 V \xrightarrow{\; \pr \;} W \ .
\end{equation*}
Obviously $\ker \partial$ does not contain any nonzero vector of the form $a \wedge b$.

Now define $\mathcal{L}' := V \oplus W$. The bracket
\begin{equation*}
    [\;,\;] : \bigwedge^2 \mathcal{L}' \xrightarrow{\; \pr \;} \bigwedge^2 V \xrightarrow{\; \partial \;} W \xrightarrow{\; \subseteq \;} \mathcal{L}'
\end{equation*}
makes $\mathcal{L}'$ into a 2-step nilpotent Lie algebra over $R$ with center $Z(\mathcal{L}') = [\mathcal{L}', \mathcal{L}'] =  W$ and $\bigwedge \mathcal{L}' \neq \ker \partial$. In fact, $\ker \partial / \bigwedge \mathcal{L}'$ is free of rank 1 over $R$.

Now let $R = \mathbb{Z}_p$ with $p \neq 2$.  Then $\mathcal{L}:=p\mathcal{L}'$ is a powerful and
torsion free, i.e.\ uniform $\mathbb{Z}_p$-Lie algebra which corresponds to a unique uniform
pro-$p$-group $G$ (of dimension $9$) by \cite{DDMS} Thm.\ 9.8. We have $pe_1\wedge pe_2 +pe_3\wedge pe_4 \in \ker \partial \setminus \bigwedge \mathcal{L}$. Thm.\ \ref{Lie-criterion1} therefore implies that $SK_1(\Lambda(G)) \neq 0$. 

In terms of generators and relations a similar example can be described as follows. Let $G$
be the pro-$p$-group generated by $x_1,\ldots , x_4,y_1,\ldots y_5$ subject to the following
relations
\begin{eqnarray*}
 [x_i,y_j]&=&  1 \mbox{ for all } 1\leq i\leq 4,\; 1\leq j \leq 5,\\
 {[y_i,y_j]}&=&
  1 \mbox{ for all } 1\leq i,j \leq 5
\end{eqnarray*} and
\begin{equation*}
    [x_i,x_j]=  \left\{
  \begin{array}{ll}
    y_1^{-p}, & \hbox{if $(i,j)=(1,2);$} \\
    y_1^p, & \hbox{if $(i,j)=(1,3);$} \\
    y_2^p, & \hbox{if $(i,j)=(1,4);$} \\
    y_3^p, & \hbox{if $(i,j)=(2,3);$} \\
    y_4^p, & \hbox{if $(i,j)=(2,4);$} \\
    y_5^p, & \hbox{if $(i,j)=(3,4).$}
  \end{array}
\right.
\end{equation*}
Then $G$ is obviously powerful, fits into an exact sequence
\begin{equation*}
    1 \longrightarrow \mathbb{Z}_p^5  \longrightarrow  G \longrightarrow \mathbb{Z}_p^4  \longrightarrow 1,
\end{equation*}
whence is torsionfree and thus uniform.
One can also show directly, using Oliver's criterion or Cor.\ \ref{S-criterion}, that this group has
non-trivial $SK_1.$ Indeed, the same kind of calculation as above shows that $x_1G^p\wedge x_2G^p +
x_3G^p\wedge x_4G^p$ belongs to $\ker(\partial:\bigwedge^2 G/G^p \to G^p/[G^p,G]),$ but not to
$\bigwedge V =\bigwedge G/G^p.$ 

\section{Chevalley orders}

Let $F$ be any field of characteristic zero and let $\mathfrak{g}$ be an $F$-split reductive Lie algebra over $F$ (\cite{B-LL} Chap.\ VIII \S2.1) with center $\mathfrak{z}$. We fix an $F$-split Cartan subalgebra $\mathfrak{h} \subseteq \mathfrak{g}$ and let $\Phi$ denote the corresponding (reduced) root system viewed as a subset of the dual vector space $\mathfrak{h}^\ast$. By \cite{B-LL} Chap.\ VIII \S2.2 Thm.\ 1 we have for any root $\alpha$:
\begin{itemize}
  \item[--] The root space $\mathfrak{g}^\alpha = \{x \in \mathfrak{g} : [h,x] = \alpha(h)x\ \textrm{for any $h \in \mathfrak{h}$}\}$ is one dimensional.
  \item[--] $\mathfrak{h}_\alpha := [\mathfrak{g}^\alpha, \mathfrak{g}^{-\alpha}]$ is one dimensional; it contains a unique element $H_\alpha$ such that $\alpha(H_\alpha) = 2$.
\end{itemize}
Obviously $H_{-\alpha} = - H_\alpha$. According to \cite{B-LL} Chap.\ VIII \S4.4 there always exists a Chevalley system for $(\mathfrak{g}, \mathfrak{h})$: This is a family of nonzero vectors $X_\alpha \in \mathfrak{g}^\alpha$ such that
\begin{equation*}
    [X_\alpha, X_{-\alpha}] = - H_\alpha
\end{equation*}
and which satisfies one additional condition whose explicit form is not needed in the following. We introduce the coroot lattice
\begin{equation*}
    Q^\vee := \sum_{\alpha \in \Phi} \mathbb{Z} H_\alpha \subseteq \mathfrak{h}
\end{equation*}
as well as the coweight lattice
\begin{equation*}
    P^\vee := \{h \in \sum_{\alpha \in \Phi} \mathbb{Q} H_\alpha : \beta(h) \in \mathbb{Z}\ \textrm{for any $\beta \in \Phi$} \} \subseteq \mathfrak{h}
\end{equation*}
of the root system $\Phi$. Throughout this section we fix a $\mathbb{Z}$-lattice $\mathfrak{h}_{\mathbb{Z}} \subseteq \mathfrak{h}$ such that
\begin{equation*}
    Q^\vee \subseteq \mathfrak{h}_{\mathbb{Z}} \subseteq P^\vee \oplus \mathfrak{z} \ ,
\end{equation*}
and we put
\begin{equation*}
    \mathfrak{g}_{\mathbb{Z}} := \mathfrak{h}_{\mathbb{Z}} + \sum_{\alpha \in \Phi} \mathbb{Z} X_\alpha \subseteq \mathfrak{g} \ .
\end{equation*}
By \cite{B-LL} Chap.\ VIII \S2.4 Prop.\ 8 and \S12.7 the $\mathbb{Z}$-submodule $\mathfrak{g}_{\mathbb{Z}}$ of $\mathfrak{g}$ in fact is a $\mathbb{Z}$-Lie subalgebra (a so called Chevalley order of $\mathfrak{g}$). For any (commutative) ring $R$ we then have the $R$-Lie algebra $\mathfrak{g}_R := R \otimes_{\mathbb{Z}} \mathfrak{g}_{\mathbb{Z}}$. We also put $\mathfrak{h}_R := R \otimes_{\mathbb{Z}} \mathfrak{h}_{\mathbb{Z}}$.

In the following we always assume that the integers $2$ and $3$ are invertible in the ring $R$. We view the Lie bracket as a linear map $[\;,\,] : \bigwedge^2 \mathfrak{g}_R \longrightarrow \mathfrak{g}_R$ on the exterior square and define
\begin{equation*}
    \bigwedge \mathfrak{g}_R := \  < x \wedge y \in \bigwedge^2 \mathfrak{g}_R : [x,y] = 0 >_R \ .
\end{equation*}
We aim at showing that
\begin{equation*}
    \bigwedge \mathfrak{g}_R = \ker [\;,\;]
\end{equation*}
holds true.

\begin{rem}
For $R = \mathbb{C}$ and $\mathfrak{g}$ semisimple this identity was shown by Kostant in \cite{Ko} Cor.\ 5.1. But his proof relies on representation theory and therefore cannot be made to work integrally.
\end{rem}

Here is a list of relevant basic facts:
\begin{itemize}
  \item[1.] $\bigwedge^2 \mathfrak{h}_R \subseteq \bigwedge \mathfrak{g}_R$ (obvious).
  \item[2.] $[X_\alpha, X_\beta] = 0$ and hence $X_\alpha \wedge X_\beta \in \bigwedge \mathfrak{g}_R$ if $\alpha + \beta \not\in \Phi \cup \{0\}$.
  \item[3.] Each root induces a surjective linear form $\alpha : \mathfrak{h}_R \longrightarrow R$ (since $\alpha(H_\alpha) = 2$ is invertible in
  $R$).
\end{itemize}
The latter implies that $\mathfrak{h}_R = R H_\alpha \oplus \ker(\alpha)$. But for $h \in \ker(\alpha)$ we have $[h,X_\alpha] = 0$ and hence $h \wedge X_\alpha \in \bigwedge \mathfrak{g}_R$. It follows that
\begin{itemize}
  \item[4.] $h \wedge X_\alpha \equiv \frac{\alpha(h)}{2} H_\alpha \wedge X_\alpha \bmod \bigwedge \mathfrak{g}_R$
\end{itemize}
for any $h \in \mathfrak{h}_R$. We fix a basis $\Delta \subseteq \Phi$ of the root system; $\Phi^+$ denotes the corresponding subset of positive roots. The set $\{H_\alpha : \alpha \in \Delta\}$ is a basis of $\mathbb{Q} \otimes_{\mathbb{Z}} Q^\vee$. We also  need the Killing form
\begin{equation*}
    <\;, \; > \, : \mathfrak{h} \times \mathfrak{h} \longrightarrow F \ .
\end{equation*}
It has the property (cf.\ \cite{B-LL} Chap.\ VIII \S2.2. Prop.\ 2(ii)) that
\begin{equation*}
    <Q^\vee,Q^\vee> \, \subseteq \mathbb{Z} \ .
\end{equation*}
We then have
\begin{itemize}
  \item[5.] $\alpha(h) = \frac{2 <H_\alpha,h>}{<H_\alpha, H_\alpha>}$ for any $h \in \mathfrak{h}$.
\end{itemize}
Since the Killing form on $\mathfrak{h}/\mathfrak{z}$ is known to be nondegenerate (\cite{B-LL} Chap.\ VIII \S2.2 Remark 2) we obtain:
\begin{itemize}
  \item[6.] If $\beta = \sum_{\alpha \in \Delta} m_{\beta\alpha} \alpha$ (with $m_{\beta\alpha} \in \mathbb{Z}$) then
      \begin{equation*}
        H_\beta = \sum_{\alpha \in \Delta} m_{\beta\alpha} \frac{<H_\beta, H_\beta>}{<H_\alpha, H_\alpha>} H_\alpha \ .
      \end{equation*}
\end{itemize}
All the ratios $\frac{<H_\beta, H_\beta>}{<H_\alpha, H_\alpha>}$ belong to the set $\{1,2, \frac{1}{2}, 3, \frac{1}{3}\}$ (\cite{B-LL} Chap.\ VI \S1.4 Prop.\ 12(i) and Chap.\ VIII \S2.4 formula (6)) and therefore are units in $R$. In particular, the $R$-module $R \otimes_{\mathbb{Z}} Q^\vee$ is free with basis $\{H_\alpha : \alpha \in \Delta\}$.

Let $\alpha, \beta$ be any two roots such that $\alpha + \beta \in \Phi$ again is a root. We have
\begin{equation*}
    [X_\alpha, X_\beta] = N_{\alpha\beta} X_{\alpha + \beta} \qquad\textrm{for some integer $N_{\alpha\beta}$}.
\end{equation*}
It follows from \cite{B-LL} Chap.\ VIII \S2.4 Prop.\ 7 and Chap.\ VI \S1.3 Cor.\ of Prop.\ 9 that
\begin{itemize}
  \item[7.] $N_{\alpha\beta} \in \pm \{1,2,3\}$ and, in particular, is a unit in $R$.
\end{itemize}
The weight lattice of $\Phi$ is $P := \Hom(Q^\vee, \mathbb{Z})$. It contains the root lattice $Q := \sum_{\alpha \in \Phi} \mathbb{Z} \alpha$ with finite index.

\begin{lemma}\label{directfactor}
Let $\alpha, \beta \in \Phi$ be any two roots; we have:
\begin{itemize}
  \item[i.] The submodule $R\alpha + R\beta$ of the $R$-module  $R \otimes_{\mathbb{Z}} P = \Hom(Q^\vee,R)$ is a direct factor;
  \item[ii.] If $\alpha \neq \pm \beta$ then we find an element $h \in R \otimes_{\mathbb{Z}} Q^\vee \subseteq \mathfrak{h}_R$ such that $\alpha(h) = 1$ and $\beta(h) = 0$.
\end{itemize}
\end{lemma}
\begin{proof}
i. \textit{Step 1:} We assume that $\Phi$ is irreducible but not of type $A_\ell$. In this case the
order of the finite group $P/Q$ is equal to $1,2,3$, or $4$ (cf. \cite{B-LL} Plate II-X). Hence $R \otimes_{\mathbb{Z}} P = R \otimes_{\mathbb{Z}} Q$, and it suffices to show that the
nonzero elementary divisors of $\mathbb{Z}\alpha + \mathbb{Z}\beta$ in $Q$ are invertible in $R$.
Since any root is a member of some basis of the root system we may assume that $\alpha \in \Delta$.
If $\alpha = \pm \beta$ then $\mathbb{Z}\alpha + \mathbb{Z}\beta = \mathbb{Z}\alpha$ obviously is a
direct factor of $Q$. We therefore suppose that $\alpha \neq \pm \beta$. Let $\beta = \sum_{\delta
\in \Delta} m_\delta \delta$. The nonzero elementary divisors in question are $1$ and the gcd of
the integers $m_\delta$ for $\delta \neq \alpha$. A case by case inspection of the tables in
\cite{B-LL} shows that this gcd is equal to $1,2$, or $3$ (the latter only happening for the type
$G_2$).

\textit{Step 2:} We assume that $\Phi$ is irreducible of type $A_\ell$. We also may assume that $\alpha$ and $\beta$ both are positive roots. The coroot lattice can be realized as
\begin{equation*}
     Q^\vee = \{(m_1, \ldots,m_{l+1}) \in \mathbb{Z}^{l+1} : m_1 + \ldots + m_{l+1} = 0\}.
\end{equation*}
The positive roots are the maps
\begin{align*}
    \alpha_{ij} : \qquad\qquad\qquad Q^\vee & \longrightarrow \mathbb{Z} \\ (m_1, \ldots,m_{l+1}) & \longmapsto m_i - m_j
\end{align*}
for any $1 \leq i < j \leq l+1$. We leave it to the reader to explicitly check that the following facts about positive roots hold true, which by duality imply the assertion.
\begin{itemize}
  \item[--] $\alpha_{12}(Q^\vee) = 2\mathbb{Z}$ for $l=1$.
  \item[--] $\alpha(Q^\vee) = \mathbb{Z}$ for $l \geq 2$ and any $\alpha$.
  \item[--] $(\mathbb{Z} \oplus \mathbb{Z}) / (\alpha,\beta)(Q^\vee) \cong \mathbb{Z}/3\mathbb{Z}$ for $l = 2$ and any $\alpha \neq \beta$.
  \item[--] $(\alpha,\beta)(Q^\vee) \supseteq 2\mathbb{Z} \oplus 2\mathbb{Z}$ for $l \geq 3$ and any $\alpha \neq \beta$.
\end{itemize}

\textit{Step 3:} If we decompose a general $\Phi$ into its irreducible components then $P$ is the direct sum of the weight lattices of these irreducible components.

ii. By i. we have $\Hom_R(R \otimes_{\mathbb{Z}} Q^\vee,R) = R\alpha \oplus R\beta \oplus U$ and the claim follows by dualizing and identifying $R \otimes_{\mathbb{Z}} Q^\vee$ with its double $R$-dual.
\end{proof}

\begin{lemma}\label{congruence1}
For any two roots $\alpha, \beta$ such that $\gamma := \alpha + \beta \in \Phi$ we have
\begin{equation*}
    X_\alpha \wedge X_\beta \equiv \frac{N_{\alpha\beta}}{2}H_\gamma \wedge X_\gamma \bmod \bigwedge \mathfrak{g}_R \ .
\end{equation*}
\end{lemma}
\begin{proof}
The difference $ X_\alpha \wedge X_\beta - \frac{N_{\alpha\beta}}{2}H_\gamma \wedge X_\gamma$ lies in the kernel of $[\; ,\;]$. Therefore, if we find elements $A, B \in \mathfrak{g}_R$ such that
\begin{equation*}
    A \wedge B \equiv X_\alpha \wedge X_\beta - \frac{N_{\alpha\beta}}{2}H_\gamma \wedge X_\gamma \bmod \bigwedge \mathfrak{g}_R
\end{equation*}
then neccesarily $[A,B] = 0$, and the assertion is proven.

By Lemma \ref{directfactor} we find an $h \in \mathfrak{h}_R$ such that $\alpha(h) = - N_{\alpha\beta}$ and $\beta(h) = 0$. According to \cite{B-LL} Chap.\ VI \S1.3 Cor.\ of Prop.\ 9 the following cases can occur:
\begin{itemize}
  \item[I.] $2\alpha + \beta \not\in \Phi$,
  \item[II.] $\gamma' := 2\alpha + \beta \in \Phi$, but $3\alpha + \beta \not\in \Phi$, or
  \item[III.] $\gamma' := 2\alpha + \beta \in \Phi$, $\gamma'' := 3\alpha + \beta \in \Phi$, but $4\alpha + \beta \not\in \Phi$.
\end{itemize}
\textit{Case I:} Here $\alpha + \gamma \not\in \Phi$ and hence, by the choice of $h$ and using 2. and 4.,
\begin{align*}
    (X_\alpha + h) \wedge (X_\beta + X_\gamma)
      & \equiv X_\alpha \wedge X_\beta + \frac{\beta(h)}{2} H_\beta \wedge X_\beta + \frac{\gamma(h)}{2} H_\gamma \wedge X_\gamma  \bmod \bigwedge \mathfrak{g}_R \\
      & = X_\alpha \wedge X_\beta - \frac{N_{\alpha\beta}}{2} H_\gamma \wedge X_\gamma
        \ .
\end{align*}
\textit{Case II:}  Here $\alpha + \gamma' \not\in \Phi$. Using 7. we define
\begin{equation*}
    r' := \frac{N_{\alpha\gamma}}{2 N_{\alpha\beta}} = - \frac{N_{\alpha\gamma}}{2 \alpha(h)} \in R \ .
\end{equation*}
We then have
\begin{align*}
     & (X_\alpha + h) \wedge (X_\beta + X_\gamma + r' X_{\gamma'})  \\
      & \equiv X_\alpha \wedge X_\beta + X_\alpha \wedge X_\gamma + \frac{\beta(h)}{2} H_\beta \wedge X_\beta + \frac{\gamma(h)}{2} H_\gamma \wedge X_\gamma + r' \frac{\gamma'(h)}{2} H_{\gamma'} \wedge X_{\gamma'} \bmod \bigwedge \mathfrak{g}_R \\
      & = X_\alpha \wedge X_\beta + X_\alpha \wedge X_\gamma - \frac{N_{\alpha\beta}}{2} H_\gamma \wedge X_\gamma - \frac{N_{\alpha\gamma}}{2} H_{\gamma'} \wedge X_{\gamma'} \\
      & \equiv X_\alpha \wedge X_\beta - \frac{N_{\alpha\beta}}{2} H_\gamma \wedge X_\gamma \bmod \bigwedge \mathfrak{g}_R
\end{align*}
where the last congruence uses the case I. for the pair $(\alpha, \gamma)$.

\noindent\textit{Case III:} We define
\begin{equation*}
    r' := \frac{N_{\alpha\gamma}}{2 N_{\alpha\beta}},\ r'' := r' \frac{N_{\alpha\gamma'}}{3 N_{\alpha\beta}} \in R \ .
\end{equation*}
Using case I. for $(\alpha,\gamma')$ and case II. for $(\alpha,\gamma)$ we have
\begin{equation*}
    (X_\alpha + h) \wedge (X_\beta + X_\gamma + r' X_{\gamma'} + r'' X_{\gamma''}) \equiv X_\alpha \wedge X_\beta - \frac{N_{\alpha\beta}}{2} H_\gamma \wedge X_\gamma \bmod \bigwedge \mathfrak{g}_R \ .
\end{equation*}
\end{proof}

Using 1., 2., 4., and Lemma \ref{congruence1} we see that the factor $R$-module $\bigwedge^2 \mathfrak{g}_R/ \bigwedge \mathfrak{g}_R$
is generated by the elements $X_\alpha \wedge X_{-\alpha}$ and $H_\alpha \wedge X_\alpha$ for
$\alpha \in \Phi$. But $[X_\alpha,X_{-\alpha}] \in \mathfrak{h}_{\mathbb{Z}}$ and $[H_\alpha,X_\alpha] =
2 X_\alpha$. The elements $X_\alpha$ are linearly independent among each other as well as from
$\mathfrak{h}_{\mathbb{Z}}$. It follows that
\begin{equation}\label{f:A}
    \ker[\;,\;] \subseteq \ \mathfrak{A} + \bigwedge \mathfrak{g}_R \quad\textrm{with}\ \mathfrak{A} := < X_\alpha \wedge
X_{-\alpha} : \alpha \in \Phi^+ >_R \ .
\end{equation}

We recall that the integers $m_{\beta\alpha}$, for $\beta \in \Phi$ and $\alpha \in
\Delta$, are defined by
\begin{equation*}
    \beta = \sum_{\alpha \in \Delta} m_{\beta\alpha} \alpha \ .
\end{equation*}

\begin{lemma}\label{congruence2}
For any $\beta \in \Phi^+$ we have
\begin{equation*}
    X_\beta \wedge X_{-\beta} \equiv \sum_{\alpha \in \Delta} m_{\beta\alpha} \frac{<H_\beta, H_\beta>}{<H_\alpha, H_\alpha>} X_\alpha \wedge X_{-\alpha} \bmod \bigwedge \mathfrak{g}_R \ .
\end{equation*}
\end{lemma}
\begin{proof}
By an inductive argument it suffices to show that
\begin{equation*}
    X_\gamma \wedge X_{-\gamma} \equiv  \frac{<H_\gamma, H_\gamma>}{<H_\alpha, H_\alpha>} X_\alpha \wedge X_{-\alpha} + \frac{<H_\gamma, H_\gamma>}{<H_\beta, H_\beta>} X_\beta \wedge X_{-\beta} \bmod \bigwedge \mathfrak{g}_R
\end{equation*}
holds true for any $\alpha, \beta \in \Phi^+$ such that $\gamma := \alpha + \beta \in \Phi$. First we observe that by 6. the difference
\begin{equation*}
    X_\gamma \wedge X_{-\gamma} - \frac{<H_\gamma, H_\gamma>}{<H_\alpha, H_\alpha>} X_\alpha \wedge X_{-\alpha} - \frac{<H_\gamma, H_\gamma>}{<H_\beta, H_\beta>} X_\beta \wedge X_{-\beta}
\end{equation*}
indeed lies in the kernel of $[\; , \;]$. We put
\begin{gather*}
    N_{\alpha, -\beta} := 0 \quad\textrm{if}\ \alpha - \beta \not\in \Phi, \\
    N_{\alpha, -2\beta} :=
    \begin{cases}
    N_{\alpha, -\beta} N_{\alpha - \beta, -\beta} & \textrm{if $\alpha - \beta \in \Phi$}, \\
    0 & \textrm{otherwise},
    \end{cases} \\
    N_{\alpha, -3\beta} :=
    \begin{cases}
    N_{\alpha, -\beta} N_{\alpha - \beta, -\beta} N_{\alpha - 2\beta, -\beta} & \textrm{if $\alpha - \beta, \alpha - 2\beta \in \Phi$}, \\
    0 & \textrm{otherwise}
    \end{cases}
\end{gather*}
and similarly with $\alpha$ and $\beta$ exchanged. Using Lemma \ref{directfactor} we choose $h \in \mathfrak{h}_R$ such that
\begin{equation*}
    \alpha(h) = - \beta(h) \quad\textrm{and}\quad \alpha(h) = -
    \frac{N_{-\alpha, \gamma} N_{-\beta, \gamma}}{N_{\alpha, \beta}} \ .
\end{equation*}
We now define
\begin{align*}
    A := X_\gamma + X_\alpha + X_\beta & + \frac{N_{\alpha, - \beta}}{ 2 N_{\gamma, -\beta}} X_{\alpha - \beta} +
    \frac{N_{\alpha, - 2\beta}}{6 N_{\gamma, -\beta}^2} X_{\alpha - 2 \beta} + \frac{N_{\alpha, - 3\beta}}{24 N_{\gamma, -\beta}^3} X_{\alpha - 3 \beta} \\
    & + \frac{N_{\beta, - \alpha}}{ 2 N_{\gamma, -\alpha}} X_{\beta - \alpha} +
    \frac{N_{\beta, - 2\alpha}}{6 N_{\gamma, -\alpha}^2} X_{\beta - 2 \alpha} + \frac{N_{\beta, - 3\alpha}}{24 N_{\gamma, -\alpha}^3} X_{\beta - 3 \alpha} \\
    & + \frac{N_{\alpha, \beta}^2}{N_{-\alpha, \gamma} N_{- \beta, \gamma}} \big( 1 - \frac{N_{\beta, -\alpha} N_{\beta - \alpha, - \beta}}{2 N_{\alpha, \beta} N_{-\beta, \gamma}} \big) h
\end{align*}
(if the subscript ? in $X_?$ is not a root then the corresponding summand is set to zero, and the product $N_{\beta, -\alpha} N_{\beta - \alpha, - \beta}$ is understood to be equal to zero if $\beta - \alpha \not\in \Phi$) and
\begin{equation*}
    B := X_{-\gamma} +
   \frac{N_{-\beta, \gamma}}{N_{\beta, \alpha}} X_{-\alpha} + \frac{N_{-\alpha, \gamma}}{N_{\alpha, \beta}} X_{-\beta} + h \ .
\end{equation*}
A lengthy but straightforward computation shows that
\begin{equation*}
    A \wedge B \equiv X_\gamma \wedge X_{-\gamma} -  \frac{<H_\gamma, H_\gamma>}{<H_\alpha, H_\alpha>} X_\alpha \wedge X_{-\alpha} - \frac{<H_\gamma, H_\gamma>}{<H_\beta, H_\beta>} X_\beta \wedge X_{-\beta} \bmod \bigwedge \mathfrak{g}_R
\end{equation*}
Besides 4. and Lemma \ref{congruence1} it uses the following identities:
\begin{itemize}
  \item[a)] $\frac{<H_{\alpha' + \beta'}, H_{\alpha' + \beta'}>}{<H_{\alpha'}, H_{\alpha'}>} = - \frac{N_{-\beta', \alpha' + \beta'}}{N_{\beta', \alpha'}}$ and $\frac{<H_{\alpha' + \beta'}, H_{\alpha' + \beta'}>}{<H_{\beta'}, H_{\beta'}>} = - \frac{N_{-\alpha', \alpha' + \beta'}}{N_{\alpha', \beta'}}$ whenever $\alpha' , \beta', \alpha' + \beta' \in \Phi$ (\cite{B-LL} Chap. VIII \S2.4 Lemma 4).
  \item[b)] $N_{\alpha', \beta'} = N_{-\alpha', -\beta'} = - N_{\beta' , \alpha'}$ whenever $\alpha' , \beta', \alpha' + \beta' \in \Phi$ (cf.\ \cite{B-LL} Chap. VIII \S2.4 Prop.\ 7).
  \item[c)] $\frac{N_{\beta, -\alpha} N_{\beta - \alpha, - \beta}}{ N_{\alpha, \beta} N_{-\beta, \gamma}} = \frac{N_{\alpha, -\beta} N_{\alpha - \beta, - \alpha}}{N_{\beta, \alpha} N_{-\alpha, \gamma}}$.
\end{itemize}
To establish the last identity c) we may assume that $\alpha - \beta \in \Phi$. It then reduces to the equality
\begin{equation*}
    \frac{N_{\beta - \alpha, -\beta}}{N_{\alpha - \beta, -\alpha}} = \frac{N_{-\beta, \gamma}}{N_{-\alpha, \gamma}} \ .
\end{equation*}
Using a) and b) for $\alpha, \beta$ we obtain
\begin{equation*}
    \frac{N_{-\beta, \gamma}}{N_{-\alpha, \gamma}} = - \frac{<H_\beta, H_\beta>}{<H_\alpha, H_\alpha>} \ .
\end{equation*}
On the other hand, using a) for $\alpha' := -\alpha$ and $\beta' := \alpha - \beta$ we have
\begin{equation*}
     \frac{N_{\beta - \alpha, -\beta}}{N_{\alpha - \beta, -\alpha}} = - \frac{<H_{-\beta}, H_{-\beta}>}{<H_{-\alpha}, H_{-\alpha}>} = - \frac{<H_\beta, H_\beta>}{<H_\alpha, H_\alpha>} \ .
\end{equation*}
\end{proof}

\begin{theorem}\label{equality}
If $2$ and $3$ are invertible in $R$ then $\ker [\;,\;] = \bigwedge \mathfrak{g}_R$.
\end{theorem}
\begin{proof}
By \eqref{f:A} and Lemma \ref{congruence2} it remains to consider any element $u \in \ker [\;,\;]$ of the form
\begin{equation*}
    u = \sum_{\beta \in \Phi^+} r_\beta \sum_{\alpha \in \Delta} m_{\beta\alpha} \frac{<H_\beta, H_\beta>}{<H_\alpha, H_\alpha>} X_\alpha \wedge X_{-\alpha}
\end{equation*}
with $r_\beta \in R$. We have
\begin{align*}
    0 & = \sum_{\beta \in \Phi^+} r_\beta \sum_{\alpha \in \Delta} m_{\beta\alpha} \frac{<H_\beta, H_\beta>}{<H_\alpha, H_\alpha>} [X_\alpha , X_{-\alpha}] \\
    & = - \sum_{\beta \in \Phi^+} r_\beta \sum_{\alpha \in \Delta} m_{\beta\alpha} \frac{<H_\beta, H_\beta>}{<H_\alpha, H_\alpha>} H_\alpha \\
    & = - \sum_{\alpha \in \Delta} \frac{1}{<H_\alpha, H_\alpha>} \big( \sum_{\beta \in \Phi^+} r_\beta m_{\beta\alpha} <H_\beta, H_\beta> \big) H_\alpha
\end{align*}
and hence
\begin{equation*}
    \sum_{\beta \in \Phi^+} r_\beta m_{\beta\alpha} <H_\beta, H_\beta>\; = 0 \qquad\textrm{for any $\alpha \in \Delta$}
\end{equation*}
by 6. But then also
\begin{equation*}
    u = \sum_{\alpha \in \Delta} \frac{1}{<H_\alpha, H_\alpha>} \big( \sum_{\beta \in \Phi^+} r_\beta m_{\beta\alpha} <H_\beta, H_\beta> \big) X_\alpha \wedge X_{-\alpha} = 0 \ .
\end{equation*}
\end{proof}

\begin{remar}\label{equalityrem}
Assuming again that $2$ and $3$ are invertible in $R$, let $a\in R$ be any non-zero-divisor. The $R$-Lie subalgebra
\begin{equation*}
    \mathfrak{g}_R(a):=a\mathfrak{g}_R:=\{ax| x\in \mathfrak{g}_R\}
\end{equation*}
of $\mathfrak{g}_R$ also satisfies
\begin{equation*}
    \ker [\;,\;] = \bigwedge \mathfrak{g}_R(a) \ .
\end{equation*}
This follows immediately from the commutative diagram
\begin{equation*}
    \xymatrix{
  {\bigwedge^2\mathfrak{g}_R} \ar[d]_{\bigwedge^2 a\cdot}^{\cong} \ar[r]^-{[\;,\;]} & {\mathfrak{g}_R} \ar@{^{(}->}[d]^{a^2\cdot} \\
  {\bigwedge^2\mathfrak{g}_R(a)}  \ar[r]^-{[\;,\;]} &  {\mathfrak{g}_R(a)}   }
\end{equation*}
in which the vertical maps are induced by multiplication by $a$ and $a^2$, respectively.
\end{remar}

The negative example in section \ref{s:counterex} compared to the above positive Thm.\ \ref{equality} could make the reader believe that the fundamental distinction here is between nilpotent and semisimple Lie algebras. But this not so. We consider the nilpotent $\mathbb{Z}$-Lie algebra
\begin{equation*}
    \mathfrak{n}_{\mathbb{Z}} := \sum_{\alpha \in \Phi^+} \mathbb{Z} X_\alpha
\end{equation*}
as well as its base change $\mathfrak{n}_R := R \otimes_{\mathbb{Z}} \mathfrak{n}_{\mathbb{Z}}$.

\begin{lemma}\label{decomp}
Let $\alpha_0, \alpha_1, \beta_0, \beta_1, \gamma \in \Phi^+$ be any five roots such that
\begin{equation*}
    \alpha_0 + \beta_0 = \gamma = \alpha_1 + \beta_1 \ ;
\end{equation*}
then in any of the pairs $\alpha_0 + \alpha_1, \alpha_0 + \beta_1$ and $\beta_0 + \alpha_1, \beta_0 + \beta_1$ and $\alpha_1 + \alpha_0, \alpha_1 + \beta_0$ and $\beta_1 + \alpha_0, \beta_1 + \beta_0$ at least one member is not a root.
\end{lemma}
\begin{proof}
Let $V$ be the underlying $\mathbb{R}$-vector space of the root system $\Phi$. If $V' \subseteq V$ denotes the subspace generated by $\alpha_0, \alpha_1, \beta_0, \beta_1$ then $\Phi' := \Phi \cap V'$ is a (reduced) root system in $V'$ (\cite{B-LL} Chap.\ VI \S1.1 Prop.\ 4). Obviously, $\Phi'$ is of rank $\leq 3$. Our assertion is an assertion about $\Phi'$. Moreover, allroots occurring in this assertion lie in the same irreducible component of $\Phi'$ as $\gamma$. We see that we may assume, without loss of generality, that $\Phi$ is irreducible of rank $\leq 3$. The case $A_1$ being empty we therefore need to consider the root systems $A_2, B_2, G_2$ and $A_3, B_3, C_3$. The assertion is clearly invariant under the Weyl group of $\Phi$. Hence it suffices (\cite{B-LL} Chap.\ VI \S1.3 Prop.\ 11) to treat the cases where $\gamma$ is either the highest long root or the highest short root. These are
\begin{align*}
    A_2 && \alpha_1 + \alpha_2 && \\
    B_2 && \alpha_1 + 2\alpha_2 && \alpha_1 + \alpha_2 \\
    G_2 && 3\alpha_1 + 2\alpha_2 && 2\alpha_1 + \alpha_2 \\
    A_3 && \alpha_1 + \alpha_2 + \alpha_3 && \\
    B_3 && \alpha_1 + 2\alpha_2 + 2\alpha_3 && \alpha_1 + \alpha_2 + \alpha_3 \\
    C_3 && 2\alpha_1 + 2\alpha_2 + \alpha_3 && \alpha_1 + 2\alpha_2 + \alpha_3
\end{align*}
(where the $\alpha_i$ form the basis of $\Phi$ corresponding to the Weyl chamber containing $\gamma$ and are appropriately ordered). The assertion can now be checked by a case by case inspection using the tables at the end of \cite{B-LL} Chap.\ VI. It is helpful to note that if $\gamma$ is the highest long root then all the roots $\alpha_0, \alpha_1, \beta_0, \beta_1$ neccessarily are positive. The most interesting case occurs for the highest short root in $B_3$ where the possible decompositions are
\begin{align*}
    \alpha_1 + \alpha_2 + \alpha_3 & = \alpha_1 + (\alpha_2 + \alpha_3) \\
& = (\alpha_1 + \alpha_2) + \alpha_3 \\
& = (\alpha_1 + \alpha_2 + 2\alpha_3) + (-\alpha_3) \\
& = (\alpha_1 + 2\alpha_2 + 2\alpha_3) + (-\alpha_2 -\alpha_3) \ .
\end{align*}
\end{proof}

\begin{proposition}\label{equality-nilp}
If $2$ and $3$ are invertible in $R$ then $\ker [\;,\;] = \bigwedge \mathfrak{n}_R$.
\end{proposition}
\begin{proof}
Using 2. it remains to show the following. Fix any $\gamma \in \Phi^+$ which can be written in at least two different ways as a sum of two positive roots. Also fix one decomposition $\gamma = \alpha_0 + \beta_0$ with $\alpha_0, \beta_0 \in \Phi^+$. Given any other such decomposition of $\gamma$ into $\alpha, \beta \in \Phi^+$ we may, by Lemma \ref{decomp}, order the pair $(\alpha, \beta)$ in such a way that $\alpha_0 + \alpha, \beta_0 + \beta \not\in \Phi^+$. Let $S_\gamma$ denote the set of all such ordered pairs (including the pair $(\alpha_0,\beta_0)$). Then $\{X_\alpha \wedge X_\beta : (\alpha, \beta) \in S_\gamma\}$ is a basis of the $R$-submodule $[\;,\;]^{-1}(RX_\gamma)$ of $\bigwedge^2 \mathfrak{n}_R$. We have to show that
\begin{equation*}
    \ker [\;,\;] \cap [\;,\;]^{-1}(RX_\gamma) \subseteq \bigwedge \mathfrak{n}_R
\end{equation*}
holds true. Let $\sum_{(\alpha,\beta) \in S_\gamma} c_{\alpha\beta} X_\alpha \wedge X_\beta \in \ker [\;,\;]$, with $c_{\alpha\beta} \in R$, so that
\begin{equation*}
    \sum_{(\alpha,\beta) \in S_\gamma} c_{\alpha\beta} N_{\alpha\beta} = 0 \ .
\end{equation*}
Since the $N_{\alpha\beta}$ are units in $R$ by 7. we may compute
\begin{multline*}
    \sum_{(\alpha,\beta) \in S_\gamma} c_{\alpha\beta} X_\alpha \wedge X_\beta = \sum_{(\alpha,\beta) \in S_\gamma} c_{\alpha\beta} N_{\alpha\beta} \frac{1}{N_{\alpha\beta}} X_\alpha \wedge X_\beta \\
= \sum_{(\alpha_0,\beta_0) \neq (\alpha,\beta) \in S_\gamma} c_{\alpha\beta} N_{\alpha\beta} (\frac{1}{N_{\alpha\beta}} X_\alpha \wedge X_\beta - \frac{1}{N_{\alpha_0\beta_0}} X_{\alpha_0} \wedge X_{\beta_0})
\end{multline*}
and
\begin{multline*}
    \frac{1}{N_{\alpha\beta}} X_\alpha \wedge X_\beta - \frac{1}{N_{\alpha_0\beta_0}} X_{\alpha_0} \wedge X_{\beta_0} = \\
(\frac{1}{N_{\alpha\beta}} X_\alpha + X_{\beta_0}) \wedge (X_\beta + \frac{1}{N_{\alpha_0\beta_0}} X_{\alpha_0}) - \frac{1}{N_{\alpha\beta} N_{\alpha_0\beta_0}} X_\alpha \wedge X_{\alpha_0} + X_\beta \wedge X_{\beta_0} \ .
\end{multline*}
Since $\alpha + \alpha_0, \beta + \beta_0 \not\in \Phi$ the three summands on the right hand side of the last identity lie in the kernel of $[\;.\;]$ and hence in $\bigwedge \mathfrak{n}_R$.
\end{proof}

We point out that for root systems of type $A_n$ we have $N_{\alpha\beta} \in \{\pm 1\}$ and therefore Prop.\ \ref{equality-nilp} holds for any ring $R$.

\section{Vanishing of $SK_1$}

In this section we consider the case $R=\mathbb{Z}_p$ for $p\neq 2,3.$ Let $\mathfrak{g}$ be a
$\mathbb{Q}_p$-split reductive Lie algebra and $\mathfrak{g}_{\mathbb{Z}}\subseteq \mathfrak{g}$
a Chevalley order. Then, for any $n\geq 1,$ the $\mathbb{Z}_p$-Lie algebra
$\mathfrak{g}_{\mathbb{Z}_p}(p^n)$ is {\it saturated} in the sense of Lazard and even {\it
powerful} in the sense of \cite{DDMS}, i.e., torsionfree as $\mathbb{Z}_p$-module and satisfying
$[\mathfrak{g}_{\mathbb{Z}_p}(p^n) ,\mathfrak{g}_{\mathbb{Z}_p}(p^n) ]\subseteq
p\mathfrak{g}_{\mathbb{Z}_p}(p^n).$ Hence, by \cite{DDMS} Thm.\ 9.10, there is a (up to
unique isomorphism) unique uniform $p$-adic Lie group $G(p^n)$ with $\mathbb{Z}_p$-Lie algebra
\begin{equation*}
    \mathcal{L}(G(p^n))=\mathfrak{g}_{\mathbb{Z}_p}(p^n) \ .
\end{equation*}
Using   Thm.\ \ref{Lie-criterion1} and Remark \ref{equalityrem}
we obtain our last main result.

\begin{theorem}\label{main}
In the above setting we have
\begin{equation*}
    SK_1(\Lambda(G(p^n))) = 0 \ .
\end{equation*}
\end{theorem}

We briefly describe how these uniform groups arise as subgroups of $\mathbb{Q}_p$-points of an algebraic
group. To this end let $\mathcal{G}$  be a split reductive group scheme over $\mathbb{Z}$ (cf.\ \cite{SGA} Exp.\ XIX Def.\ 2.7 and  Exp.\ XXII Def.\ 1.13; we recall that this includes the requirements that $\mathcal{G}$ is affine and smooth with connected fibers). Being smooth the group scheme $\mathcal{G}$ posseses a $\mathbb{Z}$-Lie algebra $\mathfrak{g}_{\mathbb{Z}}$. Its base change $\mathfrak{g} := \mathbb{Q}_p \otimes_{\mathbb{Z}} \mathfrak{g}_{\mathbb{Z}}$ is the Lie algebra of the $\mathbb{Q}_p$-split reductive algebraic group
$\mathcal{G}_{\mathbb{Q}_p} := \mathbb{Q}_p \times_{\mathbb{Z}} \mathcal{G}$ and therefore is a $\mathbb{Q}_p$-split reductive Lie algebra. At the same time $\mathfrak{g} = \Lie_{\mathbb{Q}_p}(\mathcal{G}(\mathbb{Q}_p))$ is  the Lie algebra of the (noncompact) $p$-adic Lie group $\mathcal{G}(\mathbb{Q}_p)$ of $\mathbb{Q}_p$-rational points of $\mathcal{G}$. According to \cite{Jan} \S II.1.1/11/12 the $\mathbb{Z}$-Lie algebra
$\mathfrak{g}_{\mathbb{Z}}$ indeed is a Chevalley
order of $\mathfrak{g}$. We introduce the open pro-$p$ subgroups
\begin{equation*}
    G(p^n):=\ker\big(\mathcal{G}(\mathbb{Z}_p)\to
\mathcal{G}(\mathbb{Z}/p^n)\big)
\end{equation*}
for $n \geq 1$. It is shown in \cite{HKN} Ex.\ 2.6.8 (note that they write  $\mathcal{L}^*$  for  our $\mathcal{L}$ here) that we have
\begin{equation*}
    p\mathbb{Z}_p \otimes_{\mathbb{Z}} \mathfrak{g}_{\mathbb{Z}} = \mathcal{L}(G(p)) \ .
\end{equation*}
The pro-$p$ group $G(p)$ can be identified  with the standard group $\widehat{\mathcal{G}}(p\mathbb{Z}_p)$ associated with the formal group which arises as formal completion $\widehat{\mathcal{G}}$ of $\mathcal{G}$ along its unit section. In particular $G(p)$ is uniform (cf.\ \cite{DDMS} Thm.\ 8.31). More generally the proof of loc.\ cit.\ shows that the $G(p)^{p^{n-1}} = \widehat{\mathcal{G}}(p^n\mathbb{Z}_p) = G(p^n)$ form the lower $p$-series of $G(p)$. Hence, since the $p$-power map on the group coincides with the multiplication by $p$ on the Lie algebra we obtain
\begin{equation*}
    \mathcal{L}(G(p^n) )=p^{n-1}\mathcal{L}(G(p))=p^{n}\mathrm{Lie}(\mathcal{G}_{\mathbb{Z}_p})
\end{equation*}
for any $n \geq 1$.

We finish this section with the following criterion.

\begin{remark}\label{Lie-criterion2}
For any uniform $G$ such that its Lie algebra $\Lie_{\mathbb{Q}_p}(G)$ is $\mathbb{Q}_p$-split reductive we have $SK_1(\Lambda(G)) = 0$ if and only if $\bigwedge^2 \mathcal{L}(G) / \bigwedge \mathcal{L}(G)$ is torsion free.
\end{remark}
\begin{proof}
Let $\mathfrak{g} := \Lie_{\mathbb{Q}_p}(G)$ and $\mathcal{L} := \mathcal{L}(G)$. By Thm.\ \ref{Lie-criterion1} and Thm.\ \ref{equality} the vanishing of $SK_1(\Lambda(G))$ is equivalent to the inclusion $\bigwedge \mathfrak{g} \cap \bigwedge^2 \mathcal{L} \subseteq \bigwedge \mathcal{L}$. But
\begin{equation*}
    (\bigwedge \mathfrak{g} \cap \bigwedge^2 \mathcal{L}) / \bigwedge \mathcal{L} = \tors (\bigwedge^2 \mathcal{L} / \bigwedge \mathcal{L}) \ .
\end{equation*}
\end{proof}

\end{document}